\documentclass[11pt]{article}%
\usepackage{amsfonts}
\usepackage{color}
\usepackage{amsmath}
\usepackage{fullpage}
\usepackage{amssymb}
\usepackage{graphicx}
\usepackage{blindtext}
\usepackage{titlesec}
\usepackage{authblk}
\usepackage{hyperref}%
\providecommand{\U}[1]{\protect\rule{.1in}{.1in}}
\newtheorem{theorem}{Theorem}

\newtheorem{definition}[theorem]{Definition}

\newtheorem{lemma}[theorem]{Lemma}

\newtheorem{proposition}[theorem]{Proposition}
\newtheorem{remark}[theorem]{Remark}

\newenvironment{proof}[1][Proof]{\noindent\textbf{#1.} }{\ \rule{0.5em}{0.5em}}

\newcommand{\R}{\mathbb{R}^d}

\newcommand{\norm}[1]{\left\| #1 \right\|}

\newcommand{\inner}[2]{\left< #1 , #2 \right>}

\newcommand{\Renyi}{\mathop{\mathrm{D}_{\rm{p}}}\nolimits}

\renewcommand{\div}{\operatorname{div}}
\let\originalleft\left
\let\originalright\right
\renewcommand{\left}{\mathopen{}\mathclose\bgroup\originalleft}
\renewcommand{\right}{\aftergroup\egroup\originalright}
\begin{document}

\title{Well-posedness and $L^1-L^p$ Smoothing Effect of the Porous Media Equation under Poincar\'e Inequality}
\author[1,2]{Lukang Sun\thanks{Email: \texttt{lukang.sun@tum.de} }}

\affil[1]{Technical University of Munich, School of Computation, Information and Technology, Department of Mathematics, Munich, Germany}
\affil[2]{Munich Center for Machine Learning, Munich, Germany}
\date{\today}
\maketitle

\begin{abstract}
We study the Cauchy problem for a weighted porous medium equation on $\R$
associated with a Gibbs probability measure $\pi=e^{-V}$. Under a Poincar\'e
inequality for $\pi$ and the convexity assumption on $V$, we prove
well-posedness and uniqueness of non-negative weak solutions with initial data
in $L^1(\R,\pi)$. We also establish an $L^1$--$L^p$ smoothing effect at every
positive time. More precisely, for every admissible $p>1$, we show that the
logarithm of the ratio between the $L^p(\R,\pi)$ norm of the solution and its
conserved $L^1(\R,\pi)$ mass first decays at a super-exponential rate and then
decays exponentially to zero. In particular, even if the initial datum belongs
only to $L^1(\R,\pi)$, the solution belongs to $L^p(\R,\pi)$ for every finite
$p>1$ and every $t>0$.
\end{abstract}

	


\section{Introduction}

We consider the following weighted porous medium equation on $\R$:
\begin{equation}\label{eq:1}
    \partial_t\mu=\mathcal L\mu^{1+\beta},
    \qquad \beta>0,
\end{equation}
where
\[
    \mathcal L
    :=
    \frac{1}{1+\beta}
    \left(
    \Delta-\langle\nabla V,\nabla\rangle
    \right).
\]
Here $V$ is the potential associated with the Gibbs probability measure
$\pi=e^{-V}$, and throughout the paper we assume, for simplicity, that
$V$ is smooth on $\R$. When $V$ is constant, equation \eqref{eq:1} reduces to
the classical porous medium equation.

It is useful to rewrite the equation in terms of the probability density
\[
    \rho:=\mu\pi .
\]
A direct computation shows that $\rho$ satisfies the continuity equation
\begin{equation}\label{eq:6}
    \partial_t\rho
    =
    \nabla\cdot\left(\rho v_t^\beta\right),
\end{equation}
where
\[
    v_t^\beta
    :=
    \frac{1}{\beta}
    \nabla\left(\frac{\rho_t}{\pi}\right)^\beta
    =
    \left(\frac{\rho_t}{\pi}\right)^\beta
    \nabla\log\left(\frac{\rho_t}{\pi}\right),
    \qquad \beta>0.
\]
In the limiting case $\beta=0$, equation \eqref{eq:6} becomes the
Fokker--Planck equation associated with Langevin dynamics, which is the
$L^2$-Wasserstein gradient flow of the Kullback--Leibler divergence, or
Boltzmann relative entropy. More generally, equation \eqref{eq:6} also admits a
gradient-flow interpretation and can be viewed as the gradient flow of an
$\alpha$-divergence; see, for instance, \cite[Chapter 24]{villani2008optimal}.

The goal of this paper is to study the well-posedness and uniqueness of the
Cauchy problem for \eqref{eq:1} with non-negative initial data
$\varsigma\in L^1(\R,\pi)$, and to establish an $L^1$--$L^p$ smoothing effect
for positive times. The main result is the following.
\begin{theorem}\label{thm:f1}
Assume that $V\in C^\infty(\R)$ is convex, namely
\[
    \nabla^2V\geq0,
\]
and that $\pi=e^{-V}\in\mathcal P(\R)$ satisfies the Poincar\'e inequality
\eqref{eq:pc} with constant $\lambda>0$. Let $\beta>0$ and let
\[
    \varsigma\in L^1(\R,\pi),
    \qquad
    \varsigma\geq0.
\]
Then the Cauchy problem
\begin{equation}\label{eq:re3}
    \begin{aligned}
        &\partial_t\mu
        =\mathcal L\mu^{1+\beta}=
        \frac{1}{1+\beta}
        L\mu^{1+\beta},
        \qquad
         L:=\Delta-\langle\nabla V,\nabla\rangle,\\
        &\mu(x,0)=\varsigma(x),
    \end{aligned}
\end{equation}
admits a unique non-negative weak solution in the sense of Definition
\ref{def:weak-solution}. Moreover, the solution satisfies
\[
    \mu\in\Gamma,
\]
where
\[
    \begin{aligned}
        \Gamma:=\{\mu:\;&
        \mu\in C([0,\infty),L^1(\R,\pi))
        \cap L^\infty(\tau,\infty;L^p(\R,\pi)),
        \quad \forall p\in(0,\infty),\ \forall \tau>0,\\
        &
        \partial_t\mu\in L^\infty(\tau,\infty;L^1(\R,\pi)),
        \quad \forall \tau>0,\\
        &
        \|\mu(t)\|_{L^1(\R,\pi)}
        =
        \|\varsigma\|_{L^1(\R,\pi)},
        \quad \forall t\geq0,\\
        &
        \nabla\Psi(\mu)\in L^2(\tau,\infty;L^2(\R,\pi)),
        \quad
        \mathcal L\Psi(\mu)\in L^\infty(\tau,\infty;L^1(\R,\pi)),
        \quad \forall \tau>0
        \}.
    \end{aligned}
\]
If, in addition, $\varsigma$ is locally strictly positive, in the sense that
for every compact set $Z\Subset\R$ there exists $c_Z>0$ such that
\[
    \varsigma(x)\geq c_Z
    \quad\text{for a.e. }x\in Z,
\]
then
\[
    \mu\in C^\infty(\R\times(0,\infty)),
\]
and the equation is satisfied classically for every positive time.

Let
\[
    M:=\|\varsigma\|_{L^1(\R,\pi)}.
\]
If $M=0$, then $\mu\equiv0$. If $M>0$, then for every
$p>1$ satisfying
\[
    p+\beta\geq2,
\]
the quantity
\[
    \Renyi(\mu(t))
    :=
    \log\left(\frac{\|\mu(t)\|_{L^p(\R,\pi)}}{M}\right)
\]
is finite for every $t>0$ and satisfies
\[
    \Renyi(\mu(t))\geq0.
\]
Furthermore, with
\[
    t_0:=
    \inf\left\{
    t\geq0:
    \Renyi(\mu(t))\leq\frac{p-1}{p}
    \right\},
\]
we have
\[
    \Renyi(\mu(t))
    \leq
    \begin{cases}
    -\dfrac{1}{\beta}
    \log\left(
    e^{-\beta \Renyi(\varsigma)}
    +
    \dfrac{2\beta\lambda(p-1)}{(p+\beta)^2}
    M^\beta t
    \right),
    & 0<t<t_0,\\[2ex]
    e^{-\frac{2p\lambda M^\beta(t-t_0)}{(p+\beta)^2}}
    \Renyi(\mu(t_0)),
    & t\geq t_0.
    \end{cases}
\]
Here, if $\varsigma\notin L^p(\R,\pi)$, we use the convention
\[
    \Renyi(\varsigma)=+\infty,
    \qquad
    e^{-\beta\Renyi(\varsigma)}=0.
\]
Consequently, for every $t>0$,
\[
    \mu(t)\in L^p(\R,\pi),
    \qquad
    \forall p\in(1,\infty).
\]
\end{theorem}
 Notably, in the above theorem, we do not require the initial data $\varsigma\in L^{p}(\R,\pi)$ and thus $\Renyi(\varsigma)$ could be $+\infty$. If so, we have 
\begin{equation}
    \Renyi(\mu(t)) \leq \begin{cases}- \frac{\log\left(\frac{2\beta\lambda(p-1)}{(p+\beta)^2}\norm{\varsigma}^\beta_{L^1(\R,\pi)}t\right)}{\beta} & \text { if } t\in (0,t_0),
    \\ e^{-\frac{2p\lambda \norm{\varsigma}^\beta_{L^1(\R,\pi)} (t-t_0)}{(p+\beta)^2}}\Renyi\left(\mu(t_0)\right) & \text { if } t\geq t_0,\end{cases}
\end{equation}
and thus if $t>0$, we always have $\norm{\mu(t)}_{L^p(\R,\pi)}<\infty$. We would like to note that in the first case, where $t\in (0,t_0)$, the rate of decrease is super-exponential, whereas in the second case, when $t\geq t_0$, the rate of decrease is merely exponential. As $t\to\infty$, we observe that $\Renyi(\mu(t))\to 0$; equivalently $\norm{\mu(t)}_{L^p(\R,\pi)}$ will monotonically decrease to $\norm{\mu(t)}_{L^1(\R,\pi)}=\norm{\varsigma}_{L^1(\R,\pi)}$. By applying H\"older's inequality, we conclude that the only possible limit is $\lim_{t\to\infty}\mu(t)\equiv \norm{\varsigma}_{L^1(\R,\pi)}$.

We say that $\pi=e^{-V}$ satisfies a Poincar\'e inequality with constant
$\lambda>0$ if, for every $u\in H^1(\R,\pi)$,
\begin{equation}\label{eq:pc}
    \int_{\R}u^2\,d\pi
    -
    \left(\int_{\R}u\,d\pi\right)^2
    \leq
    \frac{1}{\lambda}
    \int_{\R}|\nabla u|^2\,d\pi .
\end{equation}
The above Poincar\'e inequality plays an important role in the proof of the well-posedness, uniqueness and the $L^1-L^p$ smoothing effect of the equation. It is well-known that, if $\pi$ satisfies the log-Sobolev inequality with constant $\lambda$, then, by a simple linearization procedure, $\pi$ fulfills also the Poincar\'e inequality with constant $\lambda$. For example, on $\mathbb{R}^d$, if $V$ is strongly $\lambda-$convex, that is $\nabla^2V\geq \lambda\mathrm{I}_d$, then $\pi$ satisfies the log-Sobolev inequality with constant $\lambda$. It is well-established that according to the Holley-Stroock perturbation lemma, if $V=\bar{V}+\psi$, here $\bar{V}$ is strongly $\lambda-$convex and $\psi\in L^\infty(\R)$, then $\pi$ will satisfies a log-Sobolev inequality with some constant $\lambda'>0$. Consequently, it also satisfies a Poincar\'e inequality. The Poincar\'e inequality is weaker than the log-Sobolev inequality, for example on $\mathbb{R}$, $\pi(x)=e^{-\|x\|^\beta}/Z_{\beta}$ satisfies a Poincar\'e inequality if and only if $\beta\geq 1$, however a log-Sobolev inequality if and only if $\beta\geq 2$, see \cite[Examples 21.19]{villani2008optimal}.

We next specify the notion of weak solution used throughout the paper. In the
main theorem, we use the normalization
\[
     L:=\Delta-\langle\nabla V,\nabla\rangle,
    \qquad
    \Psi(r):=r^{1+\beta},
\]
so that the equation is written as
\[
    \partial_t\mu=\frac{1}{1+\beta} L\Psi(\mu).
\]

\begin{definition}[Weak solution]\label{def:weak-solution}
Let $T>0$ and let $\varsigma\in L^1(\R,\pi)$ be non-negative. A non-negative
function $\mu$ is called a weak solution of
\[
    \partial_t\mu=\frac{1}{1+\beta} L\Psi(\mu),
    \qquad
    \mu(\cdot,0)=\varsigma,
\]
on $\R\times[0,T]$ if the following conditions hold:

\begin{enumerate}
    \item
    \[
        \mu\in C([0,T],L^1(\R,\pi)).
    \]

    \item
    For every $\tau>0$,
    \[
        \nabla\Psi(\mu)\in L^2(\tau,T;L^2(\R,\pi)).
    \]

    \item
    The initial datum is attained in $L^1(\R,\pi)$, namely
    \[
        \lim_{t\downarrow0}
        \|\mu(t)-\varsigma\|_{L^1(\R,\pi)}=0.
    \]

    \item
    For every $0<t_1<t_2\leq T$ and every
    $\phi\in C_c^\infty(\R\times(0,T])$, one has
    \[
        \begin{aligned}
            &\int_{t_1}^{t_2}\int_{\R}
            \left[
            -\mu\,\partial_t\phi
            +
            \frac{1}{1+\beta}
            \langle\nabla\Psi(\mu),\nabla\phi\rangle
            \right]d\pi dt +
            \int_{\R}\mu(x,t)\phi(x,t)d\pi(x)\Big|_{t_1}^{t_2}
            =
            0.
        \end{aligned}
    \]
\end{enumerate}
\end{definition}

We would like to emphasize that much of the existing literature on porous
medium type equations concerns the well-posedness and regularity of
Cauchy--Dirichlet problems, including generalized porous medium equations; see,
for instance,
\cite{dibenedetto1983continuity,dahlberg1993weak,dibenedetto2012harnack}.
The classical porous medium equation is by now well understood, both from the
viewpoint of well-posedness and from that of long-time asymptotics; we refer to
\cite{carrillo1998exponential,carrillo2000asymptotic,carrillo2014renyi,vazquez2007porous}
and the references therein.

The equation studied in the present paper differs from the classical porous
medium equation because of the presence of the weighted diffusion operator
associated with the Gibbs measure \(\pi=e^{-V}\). To the best of our knowledge,
the well-posedness and long-time behavior of the Cauchy problem \eqref{eq:1}
under only a Poincar\'e inequality have not been previously addressed.

A standard approach to the long-time behavior of nonlinear diffusion equations
is the entropy dissipation method, which often relies on functional inequalities
such as Sobolev, logarithmic Sobolev, or Nash inequalities; see, for example,
\cite{carrillo2001entropy,carrillo2003kinetic,arnold2001convex}. However, the
structure of \eqref{eq:1} does not seem to fit directly into the classes of
equations for which these methods yield the desired \(L^1\)-\(L^p\) smoothing
and decay estimates.

Weighted nonlinear diffusion equations with Wasserstein gradient-flow
structure have also been investigated in other exponent regimes. In particular, \cite{iacobelli2019weighted} studied
weighted ultrafast diffusion equations, corresponding to negative diffusion
exponents. This is different from the present work, which concerns the weighted
porous medium regime \(1+\beta>1\). For fast and very fast diffusion equations,
including equations with singular diffusivities, we refer to
\cite{herrero1985cauchy,esteban1988singular,daskalopoulos2007degenerate,bonforte2010positivity,vazquez2006smoothing}.

\paragraph{Notation.}
Let $\Omega\subset\R$ be a domain. For $p\in(0,\infty)$, we write
\[
    \|f\|_{L^p(\Omega,\pi)}
    :=
    \left(\int_{\Omega}|f|^p\,d\pi\right)^{1/p},
\]
and, for a time interval $(t_1,t_2)$,
\[
    \|f\|_{L^p(\Omega\times(t_1,t_2),\pi)}
    :=
    \left(\int_{t_1}^{t_2}\int_{\Omega}|f|^p\,d\pi dt\right)^{1/p}.
\]
When $f$ is vector-valued, $|f|$ denotes its Euclidean norm.

For non-negative integers $k,k'$ and H\"older exponents
$\delta,\delta'\in(0,1]$, we use the following norms. If
$u:Z\subset\mathbb{R}^m\to\mathbb{R}$, then
\[
    \|u\|_{k+\delta;Z}
    :=
    \sum_{0\leq |\alpha|\leq k}
    \|D^\alpha u\|_{L^\infty(Z)}
    +
    \sum_{|\alpha|=k}
    \sup_{x,y\in Z,\ x\neq y}
    \frac{|D^\alpha u(x)-D^\alpha u(y)|}{|x-y|^\delta}.
\]
If
$u:Z\times Z'\subset\mathbb{R}^m\times\mathbb{R}^{m'}\to\mathbb{R}$, then
\[
    \|u\|_{k+\delta,k'+\delta';Z\times Z'}
    :=
    \sup_{y\in Z'}\|u(\cdot,y)\|_{k+\delta;Z}
    +
    \sup_{x\in Z}\|u(x,\cdot)\|_{k'+\delta';Z'}.
\]
Here $\alpha\in\mathbb{N}_0^m$ is a multi-index and
$|\alpha|=\alpha_1+\cdots+\alpha_m$. We write
\[
    u\in C^{k+\delta,k'+\delta'}(\Omega\times\Omega')
\]
whenever
\[
    \|u\|_{k+\delta,k'+\delta';\Omega\times\Omega'}<\infty.
\]

For simplicity, throughout the derivations we often work with the rescaled
equation
\[
    \partial_t\mu=L\mu^{1+\beta},
    \qquad
    L:=\Delta-\langle\nabla V,\nabla\rangle,
\]
unless explicitly stated otherwise. This is equivalent to equation
\eqref{eq:1} after a rescaling of the time variable.

\section{Main Derivations}
This section is dedicated to proving the main theorem. For the well-posedness and uniqueness of the Cauchy problem \eqref{eq:1}, we build upon the ideas presented in \cite{vazquez2007porous}, but with several important differences that we will discuss here. In \cite{vazquez2007porous}, the Barenblatt solution (the fundamental solution) is utilized as a barrier function to derive key estimates for the solution of the classical porous media equation, which are instrumental in demonstrating that the equation is uniformly parabolic. In our situation, however, there is no direct fundamental solution available to serve as a barrier function; therefore, we develop a new barrier function, as detailed in Lemma \ref{lem:re17}. Another difference lies in the Aronson-Bénilan estimate used in \cite[Proposition 9.8]{vazquez2007porous}, which shows that the solution to the classical porous media equation is $L^\infty(\R)$ for $t>0$. Due to the presence of the potential $V$ in our equation \eqref{eq:1}, the Aronson-Bénilan type estimate can only provide a local upper bound for 
$\mu$, which is insufficient for proving the well-posedness and uniqueness of our Cauchy problem. Fortunately, we can utilize the novel arguments presented in Theorem \ref{thm:4} to address this challenge.

To establish the well-posedness of the Cauchy problem with non-negative initial data in 
$L^1(\R,\pi)$, our strategy is to first demonstrate the well-posedness of the Cauchy-Dirichlet problem, and then use its solutions to approximate the solution of the Cauchy problem. The remainder of this paper is structured along these lines.

\subsection{Preliminary Lemmas}
We first derive the following lemma by direct calculation.
\begin{lemma}\label{lem:re64}
Assume that $\mu>0$ is a classical solution of
\begin{equation}\label{eq:p9}
    \partial_t\mu=L\mu^{1+\beta},
    \qquad
    L:=\Delta-\inner{\nabla V}{\nabla},
\end{equation}
on $\R\times(0,T)$. Let
\[
    \nu:=\frac{\beta+1}{\beta}\mu^\beta .
\]
Then $\nu$ satisfies
\begin{equation}\label{eq:g36}
    \partial_t\nu=\beta\nu L\nu+\norm{\nabla\nu}^2 .
\end{equation}
\end{lemma}

\begin{proof}
Since
\[
    \mu^\beta=\frac{\beta}{\beta+1}\nu,
\]
we have
\[
    \mu^{1+\beta}
    =
    \left(\frac{\beta\nu}{\beta+1}\right)^{\frac{\beta+1}{\beta}} .
\]
By the chain rule,
\begin{equation}\label{eq:p9a}
    \partial_t\nu
    =
    (\beta+1)\mu^{\beta-1}\partial_t\mu
    =
    (\beta+1)\mu^{\beta-1}
    L\left(\frac{\beta\nu}{\beta+1}\right)^{\frac{\beta+1}{\beta}} .
\end{equation}
Moreover,
\[
    \nabla\left(\frac{\beta\nu}{\beta+1}\right)^{\frac{\beta+1}{\beta}}
    =
    \left(\frac{\beta\nu}{\beta+1}\right)^{\frac{1}{\beta}}\nabla\nu
    =
    \mu\nabla\nu,
\]
and
\[
    \Delta\left(\frac{\beta\nu}{\beta+1}\right)^{\frac{\beta+1}{\beta}}
    =
    \mu\Delta\nu
    +
    \frac{\mu}{\beta\nu}\norm{\nabla\nu}^2 .
\]
Therefore, using $L=\Delta-\inner{\nabla V}{\nabla}$, we obtain
\[
    L\left(\frac{\beta\nu}{\beta+1}\right)^{\frac{\beta+1}{\beta}}
    =
    \mu L\nu+\frac{\mu}{\beta\nu}\norm{\nabla\nu}^2 .
\]
Inserting this identity into \eqref{eq:p9a} gives
\[
    \partial_t\nu
    =
    (\beta+1)\mu^\beta L\nu
    +
    \frac{(\beta+1)\mu^\beta}{\beta\nu}\norm{\nabla\nu}^2 .
\]
Since
\[
    (\beta+1)\mu^\beta=\beta\nu,
\]
we conclude that
\[
    \partial_t\nu
    =
    \beta\nu L\nu+\norm{\nabla\nu}^2 .
\]
This proves \eqref{eq:g36}.
\end{proof}

Next, we derive a lower bound for $\nu$.
\begin{lemma}\label{lem:re17}
    Assume that $\Omega$ is an open set in $\R$. Let $\nu>0$ be smooth and satisfy equation \eqref{eq:g36} in $\Omega\times (0,T]$ with
    $\nu(x,0)\geq c>0,\forall x\in B_R(x_0)$, where $B_R(x_0)\subset \Omega$. Then we can find some constants $\epsilon,\tau>0$ only depending on $c,\beta,R,d,x_0,V$, such that 
    \[
        \begin{aligned}
            &\nu(x,t)>
            \frac{\epsilon(R^2-\norm{x-x_0}^2)}{t+\tau},
            \quad\forall (x,t)\in B_{R}(x_0)\times [0,T],\text{ or, equivalently,}\\
            &\mu(x,t)>
            \Big(\frac{\beta\epsilon(R^2-\norm{x-x_0}^2)}
            {(1+\beta)(t+\tau)}\Big)^{\frac{1}{\beta}},
            \quad\forall (x,t)\in B_{R}(x_0)\times [0,T].
        \end{aligned}
    \]
\end{lemma}

\begin{proof}
Without loss of generality, we assume $x_0=0$. Define
\[
    \mathfrak{L}u:=\partial_tu-\beta uLu,
    \qquad
    L=\Delta-\inner{\nabla V}{\nabla}.
\]
Since $\nu$ satisfies \eqref{eq:g36}, we have
\[
    \mathfrak{L}\nu
    =
    \partial_t\nu-\beta\nu L\nu
    =
    \norm{\nabla\nu}^2
    \geq 0.
\]

Let
\[
    p_\epsilon(x,t):=\frac{\delta_\epsilon(x)}{t+\tau},
    \qquad
    \delta_\epsilon(x):=\epsilon(R^2-\norm{x}^2).
\]
Then
\[
    \begin{aligned}
        \mathfrak{L}p_\epsilon(x,t)
        &=
        -\frac{\delta_\epsilon(x)}{(t+\tau)^2}
        -\beta
        \frac{\delta_\epsilon(x)L\delta_\epsilon(x)}{(t+\tau)^2}  \\
        &=
        -\frac{\delta_\epsilon(x)}{(t+\tau)^2}
        \Big(1+\beta L\delta_\epsilon(x)\Big).
    \end{aligned}
\]
Moreover,
\[
    L\delta_\epsilon(x)
    =
    \epsilon\Big(-2d+2\inner{\nabla V(x)}{x}\Big).
\]
Since $V$ is smooth on $\overline{B_R(0)}$, the quantity
\[
    M_R:=
    \sup_{x\in B_R(0)}
    \left|-2d+2\inner{\nabla V(x)}{x}\right|
\]
is finite. We choose $\epsilon>0$ small enough such that
\[
    1-\beta\epsilon M_R>0.
\]
Then
\[
    1+\beta L\delta_\epsilon(x)>0,
    \qquad \forall x\in B_R(0),
\]
and therefore
\[
    \mathfrak{L}p_\epsilon(x,t)\leq 0,
    \qquad \forall (x,t)\in B_R(0)\times [0,T].
\]

Next, choose $\tau>0$ large enough so that
\[
    \frac{\epsilon R^2}{\tau}<c.
\]
Then
\[
    p_\epsilon(x,0)
    =
    \frac{\epsilon(R^2-\norm{x}^2)}{\tau}
    <c
    \leq \nu(x,0),
    \qquad \forall x\in B_R(0).
\]
On the lateral boundary $\partial B_R(0)\times[0,T]$, we have
\[
    p_\epsilon(x,t)=0<\nu(x,t),
\]
because $\nu>0$.

We are left with verifying that we can indeed use the comparison principle~(see \cite[Section 3.1]{vazquez2007porous}). We write $\mathfrak{L}$ in divergence form as
\[
    \begin{aligned}
        \mathfrak{L}u
        &=\partial_tu-\beta uLu\\
        &=\partial_tu-\beta u\Delta u+\beta u\inner{\nabla V}{\nabla u}\\
        &=\partial_tu-\beta\div(u\nabla u)
        +\beta\norm{\nabla u}^2
        +\beta u\inner{\nabla V}{\nabla u}.
    \end{aligned}
\]
Equivalently,
\[
    \mathfrak{L}u
    =
    \partial_tu-\sum_{i=1}^d\partial_i a_i(x,t,u,\nabla u)
    -b(x,t,u,\nabla u),
\]
where
\[
    a_i(x,t,u,p_1,\ldots,p_d):=\beta u p_i,
\]
and
\[
    b(x,t,u,p_1,\ldots,p_d):=
    -\beta\sum_{i=1}^d p_i^2
    -\beta u\inner{\nabla V}{\boldsymbol{p}},
    \qquad
    \boldsymbol{p}:=(p_1,\ldots,p_d)^\top .
\]
The coefficients $a_i$ and $b$ are smooth in their arguments on every set where
$u$ is bounded above and bounded away from zero. Moreover, $\nu$ is smooth,
bounded, and strictly positive on $B_R(0)\times[0,T]$, while $p_\epsilon$ is
smooth, bounded, and strictly positive in the interior of $B_R(0)\times[0,T]$.
Since $p_\epsilon=0$ on the lateral boundary and $\nu>0$ there, any first contact
point between $\nu$ and $p_\epsilon$ must occur in the interior, where
$p_\epsilon>0$. Hence the operator is locally uniformly parabolic at such a
contact point, and the standard comparison principle applies. Thus,
\[
    \nu(x,t)>p_\epsilon(x,t)
    =
    \frac{\epsilon(R^2-\norm{x}^2)}{t+\tau},
    \qquad \forall (x,t)\in B_R(0)\times[0,T].
\]

For general $x_0$, we instead take
\[
    \delta_\epsilon(x):=\epsilon(R^2-\norm{x-x_0}^2).
\]
Then
\[
    L\delta_\epsilon(x)
    =
    \epsilon\Big(-2d+2\inner{\nabla V(x)}{x-x_0}\Big),
\]
and the same argument applies with
\[
    M_{R,x_0}:=
    \sup_{x\in B_R(x_0)}
    \left|-2d+2\inner{\nabla V(x)}{x-x_0}\right|.
\]
Finally, since
\[
    \nu=\frac{\beta+1}{\beta}\mu^\beta,
\]
the lower bound for $\mu$ follows immediately from the lower bound for $\nu$.
\end{proof}

\subsection{Cauchy-Dirichlet Problem}
We first solve the Cauchy-Dirichlet problem with smooth positive bounded initial-boundary data.

\begin{proposition}\label{prop:pp4}
For small $\epsilon_1,\epsilon_2>0$ with $\epsilon_1<1/\epsilon_2$, the following Cauchy-Dirichlet problem has a unique classical solution with values in $[\epsilon_1,1/\epsilon_2]$:
    \begin{equation}\label{eq:g73}
    \begin{aligned}
        &\partial_t\mu= L\mu^{\beta+1},\quad L:=\Delta-\inner{\nabla V}{\nabla},\\
        &\mu(x,0)=\varsigma_R(x),\quad \forall x\in B_R(0),\\
        &\mu(x,t)=\epsilon_1,\quad \forall (x,t)\in \partial B_R(0)\times [0,T],
    \end{aligned}
\end{equation}
where $\varsigma_R\in \mathcal{C}^\infty(\overline{B_R(0)})$,
$\varsigma_R(x)=\epsilon_1$ for all $x\in\partial B_R(0)$, and
$\varsigma_R(x)\in [\epsilon_1,1/\epsilon_2]$ for all $x\in B_R(0)$.
\end{proposition}

\begin{proof}
We first prove existence. Choose a smooth function
$T_\epsilon:\mathbb{R}\to (0,\infty)$ such that
\[
    0<\underline{a}\leq T_\epsilon(s)\leq \overline{a}<\infty,
    \qquad \forall s\in\mathbb{R},
\]
and
\[
    T_\epsilon(s)=(1+\beta)s^\beta,
    \qquad \forall s\in [\epsilon_1,1/\epsilon_2].
\]
Consider the uniformly parabolic problem
\[
    \begin{aligned}
        &\partial_t\mu=
        \div\Big(T_\epsilon(\mu)\nabla\mu\Big)
        -T_\epsilon(\mu)\inner{\nabla V}{\nabla\mu},\\
        &\mu(x,0)=\varsigma_R(x),\quad \forall x\in B_R(0),\\
        &\mu(x,t)=\epsilon_1,\quad \forall (x,t)\in \partial B_R(0)\times [0,T].
    \end{aligned}
\]
Since $T_\epsilon$ is smooth, positive, and bounded above and below away from
zero, this is a uniformly parabolic quasilinear problem. By the classical theory
of quasilinear parabolic equations~(see \cite[Section 3.1]{vazquez2007porous}),
it admits a classical solution.

The constant functions $\epsilon_1$ and $1/\epsilon_2$ are respectively a
subsolution and a supersolution of the above truncated problem. Since
\[
    \epsilon_1\leq \varsigma_R(x)\leq 1/\epsilon_2
    \quad\text{in }B_R(0),
\]
and the lateral boundary value is $\epsilon_1$, the comparison principle gives
\[
    \epsilon_1\leq \mu(x,t)\leq 1/\epsilon_2,
    \qquad \forall (x,t)\in B_R(0)\times[0,T].
\]
Therefore the solution never enters the region where the coefficient was
modified. Hence
\[
    T_\epsilon(\mu)=(1+\beta)\mu^\beta
    \qquad\text{on }B_R(0)\times[0,T].
\]
Consequently,
\[
    \begin{aligned}
        \partial_t\mu
        &=
        \div\Big((1+\beta)\mu^\beta\nabla\mu\Big)
        -(1+\beta)\mu^\beta\inner{\nabla V}{\nabla\mu}\\
        &=
        \Delta\mu^{\beta+1}
        -\inner{\nabla V}{\nabla\mu^{\beta+1}}
        =
        L\mu^{\beta+1}.
    \end{aligned}
\]
Thus $\mu$ solves \eqref{eq:g73} classically.

Next, we prove uniqueness. Let $\mu_1,\mu_2$ be two classical solutions of
\eqref{eq:g73}. Set
\[
    w:=\mu_1^{\beta+1}-\mu_2^{\beta+1}.
\]
Then
\[
    \partial_t(\mu_1-\mu_2)=Lw.
\]
Define
\[
    \eta(x,t):=
    \begin{cases}
        \displaystyle \int_t^T w(x,s)\,ds, & t\in (0,T),\\
        0, & t\geq T.
    \end{cases}
\]
Since $\mu_1=\mu_2=\epsilon_1$ on $\partial B_R(0)\times[0,T]$, we have
$w=0$ and hence $\eta=0$ on $\partial B_R(0)\times[0,T]$. Moreover,
$\eta(x,T)=0$ and $\mu_1(x,0)-\mu_2(x,0)=0$.

Multiplying $\partial_t(\mu_1-\mu_2)=Lw$ by $\eta$ and integrating by parts
with respect to $d\pi$ and $t$, we obtain
\begin{equation}\label{eq:g22}
    \int_0^T\int_{B_R(0)}
    \inner{\nabla w}{\nabla\eta}
    -
    (\mu_1-\mu_2)\partial_t\eta
    \,d\pi(x)dt
    =0.
\end{equation}
Since
\[
    \partial_t\eta(x,t)=-w(x,t),
    \qquad
    \nabla\eta(x,t)=\int_t^T\nabla w(x,s)\,ds,
\]
equation \eqref{eq:g22} becomes
\[
    \int_0^T\int_{B_R(0)}
    (\mu_1-\mu_2)w\,d\pi(x)dt
    +
    \int_0^T\int_{B_R(0)}
    \inner{\nabla w(x,t)}{\int_t^T\nabla w(x,s)\,ds}
    d\pi(x)dt
    =0.
\]
For the second term, we compute
\[
    \begin{aligned}
        &\int_0^T\int_{B_R(0)}
        \inner{\nabla w(x,t)}{\int_t^T\nabla w(x,s)\,ds}
        d\pi(x)dt\\
        &\qquad=
        \frac12
        \int_{B_R(0)}
        \norm{\int_0^T\nabla w(x,t)\,dt}^2
        d\pi(x).
    \end{aligned}
\]
Therefore,
\[
    \int_0^T\int_{B_R(0)}
    (\mu_1-\mu_2)
    \left(\mu_1^{\beta+1}-\mu_2^{\beta+1}\right)
    d\pi(x)dt
    +
    \frac12
    \int_{B_R(0)}
    \norm{\int_0^T\nabla
    \left(\mu_1^{\beta+1}-\mu_2^{\beta+1}\right)dt}^2
    d\pi(x)
    =0.
\]
Both terms are non-negative. Hence
\[
    (\mu_1-\mu_2)
    \left(\mu_1^{\beta+1}-\mu_2^{\beta+1}\right)=0
    \quad\text{a.e. in }B_R(0)\times(0,T).
\]
Since $r\mapsto r^{\beta+1}$ is strictly increasing on $[0,\infty)$, we conclude
that
\[
    \mu_1=\mu_2
    \quad\text{in }B_R(0)\times(0,T).
\]
Thus the classical solution is unique.
\end{proof}

We also note that the solution obtained above is smooth in both space and time.
Indeed, spatial smoothness follows from Theorem \ref{thm:apx7} and the
Schauder estimate in Theorem \ref{thm:apx8}. To apply Theorem
\ref{thm:apx8}, set
\[
    \tilde{\mu}:=\mu^{\beta+1}.
\]
Since $\mu$ is bounded above and below away from zero, the equation can be
rewritten as a uniformly parabolic equation for $\tilde{\mu}$ with smooth
coefficients. Hence the Schauder interior estimates imply higher spatial
regularity. Smoothness in the time variable then follows by differentiating the
equation repeatedly with respect to time and applying the same interior
regularity argument.

Based on the above proposition, we show that the Cauchy-Dirichlet problem with zero boundary condition also has a unique classical solution in the interior.

\begin{proposition}\label{prop:re11}
For small $\epsilon_2>0$, the following Cauchy-Dirichlet problem has a unique non-negative solution $\mu$:
    \begin{equation}\label{eq:g78}
    \begin{aligned}
        &\partial_t\mu= L\mu^{\beta+1},\quad L:=\Delta-\inner{\nabla V}{\nabla},\\
        &\mu(x,0)=\varsigma_R(x),\quad \forall x\in B_R(0),\\
        &\mu(x,t)=0,\quad \forall (x,t)\in \partial B_R(0)\times [0,T],
    \end{aligned}
\end{equation}
where $\varsigma_R\in \mathcal{C}^{\infty}(\overline{B_R(0)})$,
$\varsigma_R(x)\in (0,1/\epsilon_2]$ for all $x\in B_R(0)$, and
$\varsigma_R(x)=0$ for all $x\in\partial B_R(0)$. Moreover,
\[
    0\leq \mu\leq \frac{1}{\epsilon_2}
    \quad\text{on }B_R(0)\times[0,T],
\]
and
\[
    \mu>0
    \quad\text{in }B_R(0)\times(0,T].
\]
In particular, $\mu$ is smooth in $B_R(0)\times(0,T]$ and satisfies
\eqref{eq:g78} classically there.
\end{proposition}

\begin{proof}
For each $i\geq 1$, set
\[
    a_i:=2^{-i},
    \qquad
    \varsigma_{R,i}:=\varsigma_R+a_i.
\]
By Proposition \ref{prop:pp4}, applied with lower boundary value $a_i$ and
upper bound $1/\epsilon_2+1$, there exists a unique classical solution
$\mu_i$ of
\[
    \begin{aligned}
        &\partial_t\mu_i=L\mu_i^{\beta+1},\\
        &\mu_i(x,0)=\varsigma_{R,i}(x),\quad \forall x\in B_R(0),\\
        &\mu_i(x,t)=a_i,\quad \forall (x,t)\in \partial B_R(0)\times[0,T].
    \end{aligned}
\]
By the comparison principle,
\[
    a_i\leq \mu_i(x,t)\leq \frac{1}{\epsilon_2}+1,
    \qquad \forall (x,t)\in B_R(0)\times[0,T].
\]
Moreover, since
\[
    \varsigma_{R,i+1}\leq \varsigma_{R,i},
    \qquad
    a_{i+1}\leq a_i,
\]
another application of the comparison principle gives
\[
    \mu_{i+1}\leq \mu_i
    \qquad\text{in }B_R(0)\times[0,T].
\]
Hence the pointwise limit
\[
    \mu(x,t):=\lim_{i\to\infty}\mu_i(x,t)
\]
exists and satisfies
\[
    0\leq \mu(x,t)\leq \frac{1}{\epsilon_2}+1.
\]

We next derive an estimate independent of $i$. Let
\[
    \Psi(r):=r^{\beta+1},
    \qquad
    \Phi(r):=\frac{r^{\beta+2}}{\beta+2}.
\]
Since $\mu_i=a_i$ on $\partial B_R(0)\times[0,T]$, the function
\[
    \Psi(\mu_i)-\Psi(a_i)
\]
vanishes on the lateral boundary. Thus it is an admissible test function.
Multiplying the equation by $\Psi(\mu_i)-\Psi(a_i)$ and integrating by parts
with respect to $d\pi$, we obtain
\[
    \begin{aligned}
        &\int_{B_R(0)}
        \left[
        \Phi(\mu_i(x,t))-\Psi(a_i)\mu_i(x,t)
        \right]d\pi(x)\\
        &\quad-
        \int_{B_R(0)}
        \left[
        \Phi(\mu_i(x,0))-\Psi(a_i)\mu_i(x,0)
        \right]d\pi(x)\\
        &=
        -\int_0^t\int_{B_R(0)}
        \norm{\nabla\Psi(\mu_i)}^2d\pi(x)ds .
    \end{aligned}
\]
Since
\[
    0\leq \mu_i\leq \frac{1}{\epsilon_2}+1,
\]
the first two integrals are uniformly bounded independently of $i$. Hence
\[
    \int_0^T\int_{B_R(0)}
    \norm{\nabla\Psi(\mu_i)}^2d\pi(x)dt
    \leq C,
\]
where $C$ is independent of $i$.

Therefore, up to a subsequence,
\[
    \nabla\Psi(\mu_i)\rightharpoonup \kappa
    \quad\text{weakly in }L^2(B_R(0)\times(0,T),\pi).
\]
Since $\mu_i\to\mu$ pointwise and the sequence is uniformly bounded, the
dominated convergence theorem gives
\[
    \Psi(\mu_i)\to \Psi(\mu)
    \quad\text{strongly in }L^2(B_R(0)\times(0,T),\pi).
\]
Thus
\[
    \kappa=\nabla\Psi(\mu)
\]
in the sense of distributions. Moreover,
\[
    \Psi(\mu_i)-\Psi(a_i)\in L^2(0,T;H_0^1(B_R(0))),
\]
and since $\Psi(a_i)\to0$, we obtain
\[
    \Psi(\mu)\in L^2(0,T;H_0^1(B_R(0))).
\]
This gives the zero boundary condition in the natural trace sense.

For any test function
$\phi\in \mathcal{C}_c^\infty(B_R(0)\times[0,T))$, the solution $\mu_i$
satisfies
\[
    \int_{t_1}^{t_2}\int_{B_R(0)}
    \left[
    -\mu_i\partial_t\phi
    +
    \inner{\nabla\Psi(\mu_i)}{\nabla\phi}
    \right]d\pi(x)dt
    +
    \int_{B_R(0)}
    \mu_i\phi\,d\pi(x)\Big|_{t_1}^{t_2}
    =
    0.
\]
Passing to the limit using the strong convergence of $\mu_i$ and the weak
convergence of $\nabla\Psi(\mu_i)$, we get
\[
    \int_{t_1}^{t_2}\int_{B_R(0)}
    \left[
    -\mu\partial_t\phi
    +
    \inner{\nabla\Psi(\mu)}{\nabla\phi}
    \right]d\pi(x)dt
    +
    \int_{B_R(0)}
    \mu\phi\,d\pi(x)\Big|_{t_1}^{t_2}
    =
    0.
\]
Therefore $\mu$ is a weak solution of \eqref{eq:g78}.

We now prove interior positivity and smoothness. Let $Z\Subset B_R(0)$ be a
compact set. Since $\varsigma_R>0$ in $B_R(0)$, there exists $c_Z>0$ such that
\[
    \varsigma_R(x)\geq c_Z,
    \qquad \forall x\in Z.
\]
Applying Lemma \ref{lem:re17} locally and covering $Z$ by finitely many balls
compactly contained in $B_R(0)$, we obtain, for every $0<t_1<t_2\leq T$,
\[
    \mu_i(x,t)\geq c(Z,t_1,t_2)>0,
    \qquad \forall (x,t)\in Z\times[t_1,t_2],
\]
where the constant is independent of $i$. Together with the uniform upper bound,
this makes the equation uniformly parabolic on $Z\times[t_1,t_2]$. Hence, by
Theorem \ref{thm:apx7} and the Schauder estimates in Theorem \ref{thm:apx8},
for every multi-index $k$ there exists $\delta\in(0,1)$ such that
\[
    \norm{D^k\mu_i}_{2+\delta,1+\delta/2;Z\times[t_1,t_2]}
    \leq C,
\]
where $C$ is independent of $i$. Passing to the limit yields
\[
    \mu\in \mathcal{C}^{\infty}(B_R(0)\times(0,T])
\]
and $\mu$ satisfies \eqref{eq:g78} classically in the interior. The same lower
bound also shows that
\[
    \mu>0
    \quad\text{in }B_R(0)\times(0,T].
\]

It remains to prove uniqueness. Let $\mu_1,\mu_2$ be two solutions of
\eqref{eq:g78} with the same initial and boundary data. Set
\[
    w:=\mu_1^{\beta+1}-\mu_2^{\beta+1}.
\]
Then
\[
    \partial_t(\mu_1-\mu_2)=Lw.
\]
Define
\[
    \eta(x,t):=
    \begin{cases}
        \displaystyle \int_t^T w(x,s)\,ds, & t\in(0,T),\\
        0, & t\geq T.
    \end{cases}
\]
Since both solutions have zero boundary data, we have $\eta=0$ on
$\partial B_R(0)\times[0,T]$. Multiplying by $\eta$ and integrating by parts,
we obtain
\[
    \int_0^T\int_{B_R(0)}
    \inner{\nabla w}{\nabla\eta}
    -
    (\mu_1-\mu_2)\partial_t\eta
    \,d\pi(x)dt
    =
    0.
\]
Using
\[
    \partial_t\eta=-w,
    \qquad
    \nabla\eta(x,t)=\int_t^T\nabla w(x,s)\,ds,
\]
we get
\[
    \begin{aligned}
        &\int_0^T\int_{B_R(0)}
        (\mu_1-\mu_2)w\,d\pi(x)dt\\
        &\quad+
        \frac12\int_{B_R(0)}
        \norm{\int_0^T\nabla w(x,t)\,dt}^2d\pi(x)
        =
        0.
    \end{aligned}
\]
Both terms are non-negative. Hence
\[
    (\mu_1-\mu_2)
    \left(\mu_1^{\beta+1}-\mu_2^{\beta+1}\right)=0
    \quad\text{a.e. in }B_R(0)\times(0,T).
\]
Since $r\mapsto r^{\beta+1}$ is strictly increasing on $[0,\infty)$, this implies
\[
    \mu_1=\mu_2
    \quad\text{a.e. in }B_R(0)\times(0,T).
\]
By continuity in the interior, the equality holds everywhere in
$B_R(0)\times(0,T)$. Thus the solution is unique.
\end{proof}

The following theorem pertains to the large-time behavior of 
$\norm{\mu(t)}_{L^p(B_R(0),\pi)}$, which is crucial for demonstrating the well-posedness, uniqueness, and the 
$L^1-L^p$ smoothing effect of the equation.
\begin{theorem}
   \label{thm:4}
Assume that $\pi=e^{-V}\in\mathcal{P}(\R)$ satisfies the Poincar\'e inequality with constant $\lambda>0$, $V$ smooth and $p> 1,\beta> 0,p+\beta \geq 2$. Let
\[
    L:=\Delta-\inner{\nabla V}{\nabla}.
\]
For small $\epsilon_2>0$, let $\mu$ be the classical solution to the Cauchy-Dirichlet problem
\[
    \begin{aligned}
        &\partial_t\mu= \frac{1}{1+\beta}L\mu^{\beta+1}
        =\frac{1}{1+\beta}\left[\Delta\mu^{\beta+1}-\inner{\nabla V}{\nabla \mu^{\beta+1}}\right],\\
        &\mu(x,0)=\varsigma_R(x),\quad \forall x\in B_R(0),\\
        &\mu(x,t)=0,\quad \forall (x,t)\in \partial B_R(0)\times [0,T],
    \end{aligned}
\]
where $\varsigma_R\in \mathcal{C}^\infty(\overline{B_R(0)})$,
$\norm{\varsigma_R}_{L^1(B_R(0),\pi)}=1$,
$\varsigma_R(x)\in (0,1/\epsilon_2]$ for all $x\in B_R(0)$, and
$\varsigma_R(x)=0$ for all $x\in\partial B_R(0)$. Then $K_p(\mu(t))$ is
monotonically non-increasing in time, where
\[
    K_p(\mu(t)):=
    \frac{1}{p-1}\log\left(\int_{B_R(0)}\mu^p(t)d\pi(x)\right).
\]
Moreover, with
\[
    t_0:=\inf\{t\in [0,T]:K_p(\mu(t))\leq 1\},
\]
we have
	\begin{equation}\label{eq:15}
	K_p(\mu(t)) \leq \begin{cases}
    - \dfrac{p\log\left(e^{-\frac{\beta(p-1)}{p}K_p(\varsigma_R)}+\frac{2\beta\lambda(p-1)}{(p+\beta)^2}t\right)}{\beta(p-1)} & \text { if } 0<t<t_0,
    \\[1em]
    e^{-\frac{2p\lambda (t-t_0)}{(p+\beta)^2}}K_p(\mu(t_0)) & \text { if } t\geq t_0.
    \end{cases}
	\end{equation}
If $K_p(\varsigma_R)\leq 1$, then $t_0=0$ and only the second estimate is used.
\end{theorem}

\begin{proof}
For simplicity, write
\[
    A_p(t):=\int_{B_R(0)}\mu^p(t)d\pi(x).
\]
We first prove that $A_p(t)$ is non-increasing. Since $\mu$ solves
\[
    \partial_t\mu=\frac{1}{1+\beta}L\mu^{1+\beta},
    \qquad
    L=\Delta-\inner{\nabla V}{\nabla},
\]
we have
\[
    \begin{aligned}
        \frac{d}{dt}A_p(t)
        &=p\int_{B_R(0)}\mu^{p-1}(t)\partial_t\mu(t)d\pi(x)\\
        &=\frac{p}{1+\beta}\int_{B_R(0)}
        \mu^{p-1}(t)L\mu^{1+\beta}(t)d\pi(x).
    \end{aligned}
\]
Using integration by parts with respect to $d\pi=e^{-V}dx$, and using that
$\mu=0$ on $\partial B_R(0)$, we get
\[
    \int_{B_R(0)}
    \mu^{p-1}L\mu^{1+\beta}d\pi
    =
    -\int_{B_R(0)}
    \inner{\nabla\mu^{p-1}}{\nabla\mu^{1+\beta}}d\pi.
\]
Therefore,
\[
    \begin{aligned}
        \frac{d}{dt}A_p(t)
        &=
        -\frac{p}{1+\beta}\int_{B_R(0)}
        \inner{\nabla \mu^{p-1}(t)}{\nabla\mu^{1+\beta}(t)}d\pi(x)\\
        &=
        -\frac{p}{1+\beta}(p-1)(1+\beta)
        \int_{B_R(0)}
        \mu^{p+\beta-2}(t)\norm{\nabla\mu(t)}^2d\pi(x)\\
        &=
        -p(p-1)
        \int_{B_R(0)}
        \mu^{p+\beta-2}(t)\norm{\nabla\mu(t)}^2d\pi(x).
    \end{aligned}
\]
Since
\[
    \nabla\mu^{\frac{p+\beta}{2}}
    =
    \frac{p+\beta}{2}\mu^{\frac{p+\beta-2}{2}}\nabla\mu,
\]
we have
\[
    \mu^{p+\beta-2}\norm{\nabla\mu}^2
    =
    \frac{4}{(p+\beta)^2}
    \norm{\nabla\mu^{\frac{p+\beta}{2}}}^2.
\]
Thus
\[
    \frac{d}{dt}A_p(t)
    =
    -\frac{4p(p-1)}{(p+\beta)^2}
    \int_{B_R(0)}
    \norm{\nabla \mu^{\frac{p+\beta}{2}}(t)}^2d\pi(x)
    \leq 0.
\]
Hence $A_p(t)$, and therefore $K_p(\mu(t))$, is non-increasing.

We also need the mass dissipation property. Since $\mu^{\beta+1}\geq0$ in
$B_R(0)$ and $\mu^{\beta+1}=0$ on $\partial B_R(0)$, its outward normal
derivative is non-positive on the boundary. Therefore,
\[
    \begin{aligned}
        \frac{d}{dt}\int_{B_R(0)}\mu(t)d\pi(x)
        &=
        \frac{1}{1+\beta}
        \int_{B_R(0)}L\mu^{1+\beta}(t)d\pi(x)\\
        &=
        \frac{1}{1+\beta}
        \int_{\partial B_R(0)}
        e^{-V}\partial_n\mu^{1+\beta}(t)dS
        \leq 0.
    \end{aligned}
\]
Consequently,
\[
    \int_{B_R(0)}\mu(t)d\pi(x)
    \leq
    \int_{B_R(0)}\varsigma_Rd\pi(x)
    =
    1.
\]

Now set
\[
    q:=p+\beta.
\]
Since $q\geq2$ and $\mu=0$ on $\partial B_R(0)$, the function
$\mu^{q/2}$ has zero trace on $\partial B_R(0)$. Hence its zero extension to
$\R$ belongs to $H^1(\R)$. Applying the Poincar\'e inequality to this zero
extension gives
\begin{equation}\label{eq:re30}
\begin{aligned}
    \int_{B_R(0)}
    \norm{\nabla\mu^{\frac q2}(t)}^2d\pi(x)
    &\geq
    \lambda
    \left[
    \int_{B_R(0)}\mu^q(t)d\pi(x)
    -
    \left(\int_{B_R(0)}\mu^{q/2}(t)d\pi(x)\right)^2
    \right]\\
    &\geq
    \lambda
    \left[
    \int_{B_R(0)}\mu^q(t)d\pi(x)
    -
    \left(\int_{B_R(0)}\mu^q(t)d\pi(x)\right)^{\frac{q-2}{q-1}}
    \right].
\end{aligned}
\end{equation}
We justify the last inequality. Since
\[
    \int_{B_R(0)}\mu(t)d\pi\leq1,
\]
we may apply H\"older's inequality with respect to the sub-probability measure
$\mu(t)d\pi$. If $q>2$, then
\[
    \begin{aligned}
        \int_{B_R(0)}\mu^{q/2}(t)d\pi
        &=
        \int_{B_R(0)}\mu^{(q-2)/2}(t)\mu(t)d\pi\\
        &\leq
        \left(
        \int_{B_R(0)}
        \left(\mu^{(q-2)/2}(t)\right)^{\frac{2(q-1)}{q-2}}
        \mu(t)d\pi
        \right)^{\frac{q-2}{2(q-1)}}\\
        &=
        \left(
        \int_{B_R(0)}\mu^q(t)d\pi
        \right)^{\frac{q-2}{2(q-1)}}.
    \end{aligned}
\]
Squaring both sides gives
\[
    \left(\int_{B_R(0)}\mu^{q/2}(t)d\pi\right)^2
    \leq
    \left(\int_{B_R(0)}\mu^q(t)d\pi\right)^{\frac{q-2}{q-1}}.
\]
When $q=2$, the same estimate reduces to
\[
    \left(\int_{B_R(0)}\mu(t)d\pi\right)^2\leq1,
\]
which follows from mass dissipation.

We now assume that $A_p(t)\geq1$, equivalently $K_p(\mu(t))\geq0$. Define
\[
    B_q(t):=\int_{B_R(0)}\mu^q(t)d\pi.
\]
Since $\pi(B_R(0))\leq1$ and $q>p$, H\"older's inequality gives
\[
    A_p(t)^{q/p}
    =
    \left(\int_{B_R(0)}\mu^p(t)d\pi\right)^{q/p}
    \leq
    \int_{B_R(0)}\mu^q(t)d\pi
    =
    B_q(t).
\]
Moreover, using again the sub-probability measure
\[
    d\nu_t:=\mu(t)d\pi,
\]
we obtain
\[
    \begin{aligned}
        A_p(t)
        &=
        \int_{B_R(0)}\mu^p(t)d\pi\\
        &=
        \int_{B_R(0)}\mu^{p-1}(t)d\nu_t\\
        &=
        \int_{B_R(0)}
        \left(\mu^{q-1}(t)\right)^{\frac{p-1}{q-1}}d\nu_t\\
        &\leq
        \left(
        \int_{B_R(0)}\mu^{q-1}(t)d\nu_t
        \right)^{\frac{p-1}{q-1}}
        \left(
        \int_{B_R(0)}1\,d\nu_t
        \right)^{1-\frac{p-1}{q-1}}\\
        &\leq
        \left(
        \int_{B_R(0)}\mu^q(t)d\pi
        \right)^{\frac{p-1}{q-1}}\\
        &=
        B_q(t)^{\frac{p-1}{q-1}}.
    \end{aligned}
\]
Hence
\[
    A_p(t)^{1/(p-1)}
    \leq
    B_q(t)^{1/(q-1)}.
\]
Equivalently,
\[
    B_q(t)^{-1/(q-1)}
    \leq
    A_p(t)^{-1/(p-1)}.
\]
Therefore,
\[
    1-B_q(t)^{-1/(q-1)}
    \geq
    1-A_p(t)^{-1/(p-1)}.
\]
Since $A_p(t)\geq1$, the factor
\[
    1-A_p(t)^{-1/(p-1)}
\]
is non-negative. Combining this with $B_q(t)\geq A_p(t)^{q/p}$, we get
\[
    \begin{aligned}
        B_q(t)\left(1-B_q(t)^{-\frac{1}{q-1}}\right)
        &\geq
        A_p(t)^{q/p}
        \left(1-A_p(t)^{-\frac{1}{p-1}}\right)\\
        &=
        A_p(t)^{\frac{p+\beta}{p}}
        \left(1-A_p(t)^{-\frac{1}{p-1}}\right).
    \end{aligned}
\]
Substituting this into \eqref{eq:re30}, we obtain
\[
    \int_{B_R(0)}
    \norm{\nabla\mu^{\frac{p+\beta}{2}}(t)}^2d\pi
    \geq
    \lambda
    A_p(t)^{\frac{p+\beta}{p}}
    \left(1-A_p(t)^{-\frac{1}{p-1}}\right).
\]
Returning to the differential identity for $A_p(t)$, we conclude that
\[
     \frac{d}{dt}A_p(t)
     \leq
     -\frac{4p(p-1)\lambda}{(p+\beta)^2}
     A_p(t)^{\frac{p+\beta}{p}}
     \left(1-A_p(t)^{-\frac{1}{p-1}}\right).
\]

Since
\[
    K_p(\mu(t))=\frac{1}{p-1}\log A_p(t),
\]
we have
\[
    A_p(t)=e^{(p-1)K_p(\mu(t))}.
\]
Also,
\[
    \frac{d}{dt}K_p(\mu(t))
    =
    \frac{1}{p-1}\frac{A_p'(t)}{A_p(t)}.
\]
Dividing the previous differential inequality by $(p-1)A_p(t)$ gives
\[
    \frac{d}{dt}K_p(\mu(t))
    \leq
    -\frac{4p\lambda}{(p+\beta)^2}
    A_p(t)^{\beta/p}
    \left(1-A_p(t)^{-1/(p-1)}\right).
\]
Using
\[
    A_p(t)^{\beta/p}
    =
    e^{\frac{\beta(p-1)}{p}K_p(\mu(t))}
\]
and
\[
    A_p(t)^{-1/(p-1)}
    =
    e^{-K_p(\mu(t))},
\]
we obtain
\[
    \frac{d}{dt}K_p(\mu(t))
    \leq
    -\frac{4p\lambda}{(p+\beta)^2}
    e^{\frac{\beta(p-1)}{p}K_p(\mu(t))}
    \left(1-e^{-K_p(\mu(t))}\right),
\]
whenever $K_p(\mu(t))\geq0$.

We now distinguish two regimes. First suppose $K_p(\mu(t))\geq1$. Then
\[
    1-e^{-K_p(\mu(t))}
    \geq
    1-e^{-1}
    \geq
    \frac12.
\]
Therefore, as long as $K_p(\mu(t))\geq1$,
\begin{equation}\label{eq:eq18}
    \frac{d}{dt}K_p(\mu(t))
    \leq
    -\frac{2p\lambda}{(p+\beta)^2}
    e^{\frac{\beta(p-1)}{p}K_p(\mu(t))}.
\end{equation}
Set
\[
    \alpha:=\frac{\beta(p-1)}{p}.
\]
Multiplying \eqref{eq:eq18} by $-\alpha e^{-\alpha K_p(\mu(t))}$ gives
\[
    \frac{d}{dt}e^{-\alpha K_p(\mu(t))}
    =
    -\alpha e^{-\alpha K_p(\mu(t))}
    \frac{d}{dt}K_p(\mu(t))
    \geq
    \frac{2\beta\lambda(p-1)}{(p+\beta)^2}.
\]
Integrating from $0$ to $t<t_0$, we obtain
\[
    e^{-\alpha K_p(\mu(t))}
    \geq
    e^{-\alpha K_p(\varsigma_R)}
    +
    \frac{2\beta\lambda(p-1)}{(p+\beta)^2}t.
\]
Taking logarithms and using $\alpha=\beta(p-1)/p$, we get
\[
    K_p(\mu(t))
    \leq
    -\frac{p}{\beta(p-1)}
    \log\left(
    e^{-\frac{\beta(p-1)}{p}K_p(\varsigma_R)}
    +
    \frac{2\beta\lambda(p-1)}{(p+\beta)^2}t
    \right),
    \qquad 0<t<t_0.
\]

Next suppose $0\leq K_p(\mu(t))\leq1$. Since
\[
    1-e^{-x}\geq\frac{x}{2},
    \qquad 0\leq x\leq1,
\]
and
\[
    e^{\frac{\beta(p-1)}{p}K_p(\mu(t))}\geq1,
\]
we obtain
\begin{equation}\label{eq:eq20}
    \frac{d}{dt}K_p(\mu(t))
    \leq
    -\frac{2p\lambda}{(p+\beta)^2}K_p(\mu(t)).
\end{equation}
Applying Gr\"onwall's lemma from $t_0$ to $t$ gives
\[
    K_p(\mu(t))
    \leq
    e^{-\frac{2p\lambda(t-t_0)}{(p+\beta)^2}}
    K_p(\mu(t_0)).
\]
This estimate holds as long as $K_p(\mu(s))\geq0$ on $[t_0,t]$. If $K_p$
becomes negative at a later time, the same upper bound remains true trivially,
because the right-hand side is non-negative while $K_p(\mu(t))<0$. Therefore
\eqref{eq:15} follows.
\end{proof}

\begin{remark}\label{rmk:re14}
    As a consequence of Theorem \ref{thm:4}, we obtain the uniform estimate
    \[
        K_p(\mu(t))
        \leq
        \max\left\{
        -\frac{p}{\beta(p-1)}
        \log\left(
        \frac{2\beta\lambda(p-1)}{(p+\beta)^2}t
        \right),
        1
        \right\},
        \qquad \forall t>0.
    \]
    In particular, for every $t>0$, $\beta>0$, $p>1$, and $p+\beta\geq 2$,
    there exists a constant $C=C(p,\beta,\lambda,t)$ such that
    \[
        \norm{\mu(t)}_{L^p(B_R(0),\pi)}
        \leq C(p,\beta,\lambda,t).
    \]
    Importantly, this bound is independent of
    $\norm{\varsigma_R}_{L^p(B_R(0),\pi)}$.
\end{remark}

\begin{remark}
    Theorem \ref{thm:4} also applies to the classical porous medium equation
    on $B_R(0)$. Indeed, one may choose a smooth probability measure
    $\pi=e^{-V}$ such that $V$ is constant on $B_R(0)$ and such that $\pi$
    satisfies a Poincar\'e inequality with some constant $\lambda>0$, possibly
    depending on $R$. On $B_R(0)$, the weighted operator
    $\Delta-\inner{\nabla V}{\nabla}$ then coincides with the usual Laplacian.
    Hence the conclusion of Theorem \ref{thm:4} gives the corresponding
    estimate for the classical porous medium equation on $B_R(0)$.
\end{remark}

\subsection{Cauchy Problem}
We will now utilize the solutions obtained from the Cauchy-Dirichlet problem to construct a solution for the Cauchy problem.

\subsubsection{Cauchy Problem with Smooth Positive Bounded Initial Data}
We first consider the Cauchy problem with smooth positive bounded initial data.

\begin{proposition}\label{prop:re19}
Assume that $\pi=e^{-V}\in\mathcal{P}(\R)$ and that $V$ is smooth. For small
$\epsilon_2>0$, the following Cauchy problem has a unique bounded positive
classical solution in $\R\times(0,T]$:
\begin{equation}\label{eq:p43}
    \begin{aligned}
        &\partial_t\mu= L\mu^{\beta+1},\quad L:=\Delta-\inner{\nabla V}{\nabla},\\
        &\mu(x,0)=\varsigma(x),\quad \forall x\in \R,
    \end{aligned}
\end{equation}
where $\varsigma\in \mathcal{C}^{\infty}(\R)$ and
\[
    0<\varsigma(x)\leq \frac{1}{\epsilon_2},
    \qquad \forall x\in\R.
\]
Moreover,
\[
    0<\mu(x,t)\leq \frac{1}{\epsilon_2},
    \qquad (x,t)\in\R\times(0,T],
\]
and
\[
    \mu\in \mathcal{C}([0,T],L^1(\R,\pi)).
\]
In addition,
\[
    \norm{\mu(t)}_{L^1(\R,\pi)}
\]
is non-increasing in time,
\[
    \mu\in L^\infty(0,T;L^p(\R,\pi)),
    \qquad \forall p\in(0,\infty),
\]
and
\[
    \nabla\Psi(\mu)\in L^2(0,T;L^2(\R,\pi)),
    \qquad
    \Psi(\mu):=\mu^{1+\beta}.
\]
Finally, for every $\tau>0$,
\[
    \partial_t\mu\in L^\infty(\tau,T;L^1(\R,\pi)),
    \qquad
    L\Psi(\mu)\in L^\infty(\tau,T;L^1(\R,\pi)).
\]
\end{proposition}

\begin{proof}
Let $R_i:=2^i$. Choose functions
$\chi_i\in\mathcal{C}^{\infty}(\overline{B_{R_i}(0)})$ such that
\[
    0\leq \chi_i\leq 1,\qquad
    \chi_i=1 \text{ on }B_{R_i-1}(0),\qquad
    \chi_i>0 \text{ in }B_{R_i}(0),
\]
and
\[
    \chi_i=0 \text{ on }\partial B_{R_i}(0).
\]
Set
\[
    \varsigma_i:=\varsigma\chi_i.
\]
Then
\[
    0<\varsigma_i(x)\leq \frac{1}{\epsilon_2}
    \quad\text{in }B_{R_i}(0),
    \qquad
    \varsigma_i=0
    \quad\text{on }\partial B_{R_i}(0).
\]
By Proposition \ref{prop:re11}, there exists a unique solution $\mu_i$ of
\[
    \begin{aligned}
        &\partial_t\mu_i=L\mu_i^{\beta+1},
        \qquad \text{in }B_{R_i}(0)\times(0,T],\\
        &\mu_i(x,0)=\varsigma_i(x),
        \qquad x\in B_{R_i}(0),\\
        &\mu_i(x,t)=0,
        \qquad (x,t)\in\partial B_{R_i}(0)\times[0,T].
    \end{aligned}
\]
The comparison principle gives
\[
    0\leq \mu_i(x,t)\leq \frac{1}{\epsilon_2},
    \qquad (x,t)\in B_{R_i}(0)\times[0,T].
\]

Let $Z\Subset\R$ and $0<t_1<t_2\leq T$. For all sufficiently large $i$, we have
$Z\Subset B_{R_i}(0)$ and $\chi_i=1$ on a neighbourhood of $Z$. Since
$\varsigma>0$ and is continuous, there exists $c_Z>0$, independent of $i$, such
that
\[
    \varsigma_i(x)\geq c_Z,
    \qquad x\in Z.
\]
Applying Lemma \ref{lem:re17} locally and covering $Z$ by finitely many balls,
we obtain
\[
    \mu_i(x,t)\geq c(Z,t_1,t_2)>0,
    \qquad (x,t)\in Z\times[t_1,t_2],
\]
with a constant independent of $i$. Together with the uniform upper bound, this
makes the equation uniformly parabolic on $Z\times[t_1,t_2]$, uniformly in $i$.
Therefore, by Theorem \ref{thm:apx7} and the Schauder estimate in Theorem
\ref{thm:apx8}, for every multi-index $k$ there exist $\delta\in(0,1)$ and a
constant $C$, independent of $i$, such that
\[
    \norm{D^k\mu_i}_{2+\delta,1+\delta/2;Z\times[t_1,t_2]}
    \leq C.
\]
By a diagonal argument, there exists a function $\mu$ such that, along a
subsequence,
\[
    \mu_i\to\mu
    \quad\text{locally smoothly in }\R\times(0,T].
\]
Hence $\mu$ is positive and smooth in $\R\times(0,T]$, satisfies
\[
    \partial_t\mu=L\mu^{\beta+1},
\]
and obeys
\[
    0<\mu(x,t)\leq \frac{1}{\epsilon_2}.
\]

We now show that $\mu$ attains the initial datum. Fix $Z\Subset\R$ and
$t_2>0$. For all sufficiently large $i$, we have $\chi_i=1$ on a neighbourhood
of $Z$, and hence
\[
    \varsigma_i=\varsigma
    \quad\text{on }Z.
\]
Moreover, Lemma \ref{lem:re17} gives a positive lower bound for $\mu_i$ on
$Z\times[0,t_2]$, independent of $i$. Thus the equation is uniformly parabolic
there up to the initial time. The local Schauder estimates up to the initial
time~(see for example \cite[Theorem 1996]{lieberman1996second}) imply local compactness on $Z\times[0,t_2]$. Passing to the limit gives
\[
    \mu(x,0)=\lim_{i\to\infty}\mu_i(x,0)
    =
    \lim_{i\to\infty}\varsigma_i(x)
    =
    \varsigma(x),
\]
locally uniformly in $x$. Therefore
\[
    \mu(x,t)\to\varsigma(x)
    \quad\text{for every }x\in\R
    \quad\text{as }t\downarrow0.
\]
Since
\[
    0\leq\mu(x,t)\leq \frac{1}{\epsilon_2},
    \qquad
    0\leq\varsigma(x)\leq \frac{1}{\epsilon_2},
\]
and $\pi(\R)=1$, dominated convergence gives
\[
    \lim_{t\downarrow0}
    \norm{\mu(t)-\varsigma}_{L^1(\R,\pi)}=0.
\]
For $t>0$, the same dominated convergence argument gives continuity in
$L^1(\R,\pi)$. Thus
\[
    \mu\in\mathcal{C}([0,T],L^1(\R,\pi)).
\]

Next, we derive the energy estimate. For each $\mu_i$, testing the equation
with $\Psi(\mu_i)=\mu_i^{\beta+1}$ gives
\[
   \int_{B_{R_i}(0)}\Phi(\mu_i(x,t))d\pi(x)
   -
   \int_{B_{R_i}(0)}\Phi(\varsigma_i(x))d\pi(x)
   =
   -\int_0^t\int_{B_{R_i}(0)}
   \norm{\nabla\Psi(\mu_i)}^2d\pi(x)ds,
\]
where
\[
    \Phi(r):=\frac{r^{\beta+2}}{\beta+2}.
\]
Since
\[
    0\leq \varsigma_i\leq \frac{1}{\epsilon_2},
\]
we obtain
\[
    \int_0^t\int_{B_{R_i}(0)}
    \norm{\nabla\Psi(\mu_i)}^2d\pi(x)ds
    \leq C(t,\beta,\epsilon_2),
    \qquad 0<t\leq T,
\]
where $C$ is independent of $i$. Passing to the limit and applying Fatou's
lemma gives
\[
    \int_0^t\int_{\R}
    \norm{\nabla\Psi(\mu)}^2d\pi(x)ds
    \leq C(t,\beta,\epsilon_2),
    \qquad 0<t\leq T.
\]
Hence
\[
    \nabla\Psi(\mu)\in L^2(0,T;L^2(\R,\pi)).
\]

Since
\[
    0\leq \mu\leq \frac{1}{\epsilon_2}
\]
and $\pi(\R)=1$, we immediately have
\[
    \mu\in L^\infty(0,T;L^p(\R,\pi)),
    \qquad \forall p\in(0,\infty).
\]

The $L^1$ norm is non-increasing. Indeed, applying Proposition \ref{prop:g6}
with $\mu_1=\mu$ and $\mu_2=0$ yields
\[
    \norm{\mu(t_2)}_{L^1(\R,\pi)}
    \leq
    \norm{\mu(t_1)}_{L^1(\R,\pi)},
    \qquad 0\leq t_1\leq t_2\leq T.
\]

We now estimate $\partial_t\mu$. Fix $t>0$ and $h>0$, and choose $\eta>1$ such
that
\[
    \eta^\beta t=t+h.
\]
Define
\[
    \tilde{\mu}(x,s):=\eta\mu(x,\eta^\beta s).
\]
Then $\tilde{\mu}$ solves the same equation with initial datum $\eta\varsigma$.
By Proposition \ref{prop:g6},
\[
    \norm{\tilde{\mu}(t)-\mu(t)}_{L^1(\R,\pi)}
    \leq
    \norm{\eta\varsigma-\varsigma}_{L^1(\R,\pi)}
    =
    (\eta-1)\norm{\varsigma}_{L^1(\R,\pi)}.
\]
Since $\tilde{\mu}(t)=\eta\mu(t+h)$, we get
\[
    \norm{\mu(t+h)-\mu(t)}_{L^1(\R,\pi)}
    \leq
    2(\eta-1)\norm{\varsigma}_{L^1(\R,\pi)}.
\]
Moreover,
\[
    \eta=\left(1+\frac{h}{t}\right)^{1/\beta},
\]
so
\[
    \lim_{h\downarrow0}\frac{\eta-1}{h}
    =
    \frac{1}{\beta t}.
\]
Thus
\[
    \limsup_{h\downarrow0}
    \frac{\norm{\mu(t+h)-\mu(t)}_{L^1(\R,\pi)}}{h}
    \leq
    \frac{2\norm{\varsigma}_{L^1(\R,\pi)}}{\beta t}.
\]
Since $\mu$ is classical for positive times, Fatou's lemma gives
\[
    \norm{\partial_t\mu(t)}_{L^1(\R,\pi)}
    \leq
    \frac{2\norm{\varsigma}_{L^1(\R,\pi)}}{\beta t}.
\]
Because $\partial_t\mu=L\Psi(\mu)$, we also have
\[
    L\Psi(\mu)\in L^\infty(\tau,T;L^1(\R,\pi)),
    \qquad \forall \tau>0.
\]

Finally, uniqueness follows from Proposition \ref{prop:g6}. If $\mu_1$ and
$\mu_2$ are two solutions with the same initial datum, then
\[
    \norm{\mu_1(t)-\mu_2(t)}_{L^1(\R,\pi)}
    \leq
    \norm{\mu_1(0)-\mu_2(0)}_{L^1(\R,\pi)}
    =
    0.
\]
Hence $\mu_1=\mu_2$, and the proof is complete.
\end{proof}
\begin{proposition}\label{prop:re19-poincare}
Assume, in addition to the hypotheses of Proposition \ref{prop:re19}, that
$\pi=e^{-V}$ satisfies the Poincar\'e inequality with constant $\lambda>0$.
Let $\mu$ be the solution of \eqref{eq:p43} constructed in Proposition
\ref{prop:re19}. Then, for every $p\in(0,\infty)$ and every $\tau>0$,
\[
    \mu\in L^\infty(\tau,T;L^p(\R,\pi)),
\]
with a bound depending only on
\[
    p,\beta,\lambda,\tau,T,\norm{\varsigma}_{L^1(\R,\pi)},
\]
and not on $\epsilon_2$ or on $\norm{\varsigma}_{L^p(\R,\pi)}$.

Moreover,
\[
    \nabla\Psi(\mu)\in L^2(\tau,T;L^2(\R,\pi)),
\]
with
\[
    \int_\tau^T\int_{\R}
    \norm{\nabla\Psi(\mu)}^2d\pi(x)dt
    \leq
    C(\tau,T,\beta,\lambda,\norm{\varsigma}_{L^1(\R,\pi)}),
\]
where the constant is independent of $\epsilon_2$.
\end{proposition}

\begin{proof}
Let $\mu_i$ be the Cauchy--Dirichlet approximations used in the proof of
Proposition \ref{prop:re19}. Set
\[
    M_i:=\norm{\varsigma_i}_{L^1(B_{R_i}(0),\pi)}.
\]
Then
\[
    M_i\leq \norm{\varsigma}_{L^1(\R,\pi)}.
\]
Applying Theorem \ref{thm:4} to the normalized solution
\[
    \bar{\mu}_i(x,t):=
    M_i^{-1}\mu_i(x,M_i^{-\beta}t),
\]
and then scaling back, gives, for every $p>1$ with $p+\beta\geq2$,
\[
    \norm{\mu_i(t)}_{L^p(B_{R_i}(0),\pi)}
    \leq
    C(p,\beta,\lambda,t,\norm{\varsigma}_{L^1(\R,\pi)}),
    \qquad t>0,
\]
where the constant is independent of $i$ and $\epsilon_2$. Passing to the limit
and applying Fatou's lemma yields
\[
    \norm{\mu(t)}_{L^p(\R,\pi)}
    \leq
    C(p,\beta,\lambda,t,\norm{\varsigma}_{L^1(\R,\pi)}).
\]
Taking the supremum over $t\in[\tau,T]$ gives
\[
    \mu\in L^\infty(\tau,T;L^p(\R,\pi))
\]
for all $p>1$ satisfying $p+\beta\geq2$.

If $p>1$ but $p+\beta<2$, choose $q>p$ such that $q+\beta\geq2$. Since
$\pi(\R)=1$, H\"older's inequality gives
\[
    \norm{\mu(t)}_{L^p(\R,\pi)}
    \leq
    \norm{\mu(t)}_{L^q(\R,\pi)}.
\]
Thus the same conclusion holds for every $p>1$. For $0<p\leq1$, using
$\pi(\R)=1$ and the non-increasing property of the $L^1$ norm, we have
\[
    \norm{\mu(t)}_{L^p(\R,\pi)}
    \leq
    \norm{\mu(t)}_{L^1(\R,\pi)}
    \leq
    \norm{\varsigma}_{L^1(\R,\pi)}.
\]
Therefore
\[
    \mu\in L^\infty(\tau,T;L^p(\R,\pi)),
    \qquad \forall p\in(0,\infty),\ \forall\tau>0.
\]

It remains to prove the gradient estimate independent of $\epsilon_2$. For each
$\mu_i$, the energy identity gives
\[
   \begin{aligned}
       &\int_{B_{R_i}(0)}\Phi(\mu_i(x,T))d\pi(x)
       -\int_{B_{R_i}(0)}\Phi(\mu_i(x,\tau))d\pi(x)\\
       &\qquad=
       -\int_\tau^T\int_{B_{R_i}(0)}
       \norm{\nabla\Psi(\mu_i)}^2d\pi(x)dt,
   \end{aligned}
\]
where
\[
    \Phi(r):=\frac{r^{\beta+2}}{\beta+2}.
\]
Hence
\[
    \int_\tau^T\int_{B_{R_i}(0)}
    \norm{\nabla\Psi(\mu_i)}^2d\pi(x)dt
    \leq
    \int_{B_{R_i}(0)}\Phi(\mu_i(x,\tau))d\pi(x).
\]
Using the $L^{\beta+2}$ estimate already obtained at time $\tau$, we get
\[
    \int_\tau^T\int_{B_{R_i}(0)}
    \norm{\nabla\Psi(\mu_i)}^2d\pi(x)dt
    \leq
    C(\tau,T,\beta,\lambda,\norm{\varsigma}_{L^1(\R,\pi)}),
\]
with a constant independent of $i$ and $\epsilon_2$. Passing to the limit and
using Fatou's lemma yields
\[
    \int_\tau^T\int_{\R}
    \norm{\nabla\Psi(\mu)}^2d\pi(x)dt
    \leq
    C(\tau,T,\beta,\lambda,\norm{\varsigma}_{L^1(\R,\pi)}).
\]
This proves the proposition.
\end{proof}

The following proposition establishes an $L^1(\R,\pi)$ contraction property.

\begin{proposition}\label{prop:g6}
Let
\[
    L:=\Delta-\inner{\nabla V}{\nabla},
    \qquad
    \Psi(r):=r^{1+\beta}.
\]
Let $\mu_1,\mu_2$ be two non-negative classical solutions of
\[
    \partial_t\mu=L\Psi(\mu)
\]
on $\R\times(0,T]$, such that
\[
    \mu_i\in \mathcal{C}([0,T],L^1(\R,\pi)),
    \qquad i=1,2,
\]
and
\[
    \int_0^T\int_{\R}\norm{\nabla\Psi(\mu_i)}^2d\pi(x)dt<\infty,
    \qquad i=1,2.
\]
Then, for every $0\leq t_1\leq t_2\leq T$,
\[
    \norm{(\mu_1(t_2)-\mu_2(t_2))^+}_{L^1(\R,\pi)}
    \leq
    \norm{(\mu_1(t_1)-\mu_2(t_1))^+}_{L^1(\R,\pi)}.
\]
Consequently,
\[
    \norm{\mu_1(t_2)-\mu_2(t_2)}_{L^1(\R,\pi)}
    \leq
    \norm{\mu_1(t_1)-\mu_2(t_1)}_{L^1(\R,\pi)}.
\]
In particular, if $\mu_1(0)=\mu_2(0)$ in $L^1(\R,\pi)$, then
$\mu_1=\mu_2$ on $\R\times[0,T]$.
\end{proposition}

\begin{proof}
Set
\[
    y:=\mu_1-\mu_2,
    \qquad
    w:=\Psi(\mu_1)-\Psi(\mu_2).
\]
Since $\Psi(r)=r^{1+\beta}$ is increasing on $[0,\infty)$, we have
\[
    w>0 \Longleftrightarrow y>0,
    \qquad
    w=0 \Longleftrightarrow y=0.
\]
Subtracting the two equations gives
\[
    \partial_t y=Lw.
\]

Let $\theta_\delta\in C^1(\mathbb{R})$ be such that
\[
    0\leq \theta_\delta\leq 1,
    \qquad
    \theta_\delta(s)=0 \text{ for } s\leq0,
    \qquad
    \theta_\delta'(s)\geq0,
\]
and
\[
    \theta_\delta(s)\to \operatorname{sign}_0^+(s)
    \quad\text{as }\delta\downarrow0.
\]
Here
\[
    \operatorname{sign}_0^+(s)=
    \begin{cases}
        1, & s>0,\\
        0, & s\leq0.
    \end{cases}
\]
Let $\zeta_1\in C_c^\infty(\R)$ satisfy
\[
    0\leq \zeta_1\leq1,
    \qquad
    \zeta_1(x)=1 \text{ if } \norm{x}\leq1,
    \qquad
    \zeta_1(x)=0 \text{ if } \norm{x}\geq2,
\]
and define
\[
    \zeta_n(x):=\zeta_1(x/n).
\]
Then
\[
    0\leq \zeta_n\leq1,
    \qquad
    \zeta_n\uparrow1,
    \qquad
    \norm{\nabla\zeta_n}_{L^\infty}\leq \frac{C}{n}.
\]

Fix $0<t_1<t_2\leq T$. Multiplying
\[
    \partial_t y=Lw
\]
by $\theta_\delta(w)\zeta_n$ and integrating over
$\R\times(t_1,t_2)$ gives
\[
    \int_{t_1}^{t_2}\int_{\R}
    \partial_t y\,\theta_\delta(w)\zeta_n\,d\pi dt
    =
    \int_{t_1}^{t_2}\int_{\R}
    Lw\,\theta_\delta(w)\zeta_n\,d\pi dt.
\]
Using integration by parts with respect to $d\pi=e^{-V}dx$, we obtain
\[
    \begin{aligned}
        \int_{t_1}^{t_2}\int_{\R}
        Lw\,\theta_\delta(w)\zeta_n\,d\pi dt
        &=
        -\int_{t_1}^{t_2}\int_{\R}
        \inner{\nabla w}{\nabla(\theta_\delta(w)\zeta_n)}d\pi dt\\
        &=
        -\int_{t_1}^{t_2}\int_{\R}
        \theta_\delta'(w)\norm{\nabla w}^2\zeta_n\,d\pi dt\\
        &\quad
        -\int_{t_1}^{t_2}\int_{\R}
        \theta_\delta(w)\inner{\nabla w}{\nabla\zeta_n}d\pi dt.
    \end{aligned}
\]
Since $\theta_\delta'\geq0$, the first term on the right-hand side is
non-positive. Hence
\[
    \int_{t_1}^{t_2}\int_{\R}
    \partial_t y\,\theta_\delta(w)\zeta_n\,d\pi dt
    \leq
    -\int_{t_1}^{t_2}\int_{\R}
    \theta_\delta(w)\inner{\nabla w}{\nabla\zeta_n}d\pi dt.
\]
Letting $\delta\downarrow0$, and using the fact that $w$ and $y$ have the same
sign, we get
\[
    \int_{t_1}^{t_2}\int_{\R}
    \partial_t y\,\operatorname{sign}_0^+(y)\zeta_n\,d\pi dt
    \leq
    -\int_{t_1}^{t_2}\int_{\R}
    \operatorname{sign}_0^+(w)
    \inner{\nabla w}{\nabla\zeta_n}d\pi dt.
\]
Since
\[
    \partial_t y\,\operatorname{sign}_0^+(y)
    =
    \partial_t y^+,
\]
the left-hand side becomes
\[
    \int_{\R}y^+(x,t_2)\zeta_n(x)d\pi(x)
    -
    \int_{\R}y^+(x,t_1)\zeta_n(x)d\pi(x).
\]
Therefore,
\[
    \begin{aligned}
        &\int_{\R}y^+(x,t_2)\zeta_n(x)d\pi(x)
        -
        \int_{\R}y^+(x,t_1)\zeta_n(x)d\pi(x)\\
        &\leq
        \left|
        \int_{t_1}^{t_2}\int_{\R}
        \operatorname{sign}_0^+(w)
        \inner{\nabla w}{\nabla\zeta_n}d\pi dt
        \right|.
    \end{aligned}
\]
Using $0\leq \operatorname{sign}_0^+(w)\leq1$ and Cauchy's inequality, we have
\[
    \begin{aligned}
        &\left|
        \int_{t_1}^{t_2}\int_{\R}
        \operatorname{sign}_0^+(w)
        \inner{\nabla w}{\nabla\zeta_n}d\pi dt
        \right|\\
        &\leq
        \left(
        \int_{t_1}^{t_2}\int_{\R}\norm{\nabla\zeta_n}^2d\pi dt
        \right)^{1/2}
        \left(
        \int_{t_1}^{t_2}\int_{\R}\norm{\nabla w}^2d\pi dt
        \right)^{1/2}.
    \end{aligned}
\]
Moreover,
\[
    \nabla w
    =
    \nabla\Psi(\mu_1)-\nabla\Psi(\mu_2),
\]
and therefore
\[
    \int_0^T\int_{\R}\norm{\nabla w}^2d\pi dt
    \leq
    2\sum_{i=1}^2
    \int_0^T\int_{\R}\norm{\nabla\Psi(\mu_i)}^2d\pi dt
    <\infty.
\]
Also,
\[
    \int_{t_1}^{t_2}\int_{\R}\norm{\nabla\zeta_n}^2d\pi dt
    \leq
    \frac{C(t_2-t_1)}{n^2}.
\]
Hence
\[
    \int_{\R}y^+(x,t_2)\zeta_n(x)d\pi(x)
    \leq
    \int_{\R}y^+(x,t_1)\zeta_n(x)d\pi(x)
    +
    \frac{C}{n}.
\]
Letting $n\to\infty$ and using monotone convergence gives
\[
    \int_{\R}y^+(x,t_2)d\pi(x)
    \leq
    \int_{\R}y^+(x,t_1)d\pi(x).
\]
Thus
\[
    \norm{(\mu_1(t_2)-\mu_2(t_2))^+}_{L^1(\R,\pi)}
    \leq
    \norm{(\mu_1(t_1)-\mu_2(t_1))^+}_{L^1(\R,\pi)}.
\]
By the $L^1$-continuity of $\mu_1,\mu_2$, the same estimate also holds for
$t_1=0$.

Exchanging the roles of $\mu_1$ and $\mu_2$ gives
\[
    \norm{(\mu_2(t_2)-\mu_1(t_2))^+}_{L^1(\R,\pi)}
    \leq
    \norm{(\mu_2(t_1)-\mu_1(t_1))^+}_{L^1(\R,\pi)}.
\]
Adding the two inequalities yields
\[
    \norm{\mu_1(t_2)-\mu_2(t_2)}_{L^1(\R,\pi)}
    \leq
    \norm{\mu_1(t_1)-\mu_2(t_1)}_{L^1(\R,\pi)}.
\]
If $\mu_1(0)=\mu_2(0)$ in $L^1(\R,\pi)$, then the right-hand side is zero
with $t_1=0$, and therefore $\mu_1=\mu_2$ for all $t\in[0,T]$.
\end{proof}
\begin{remark}\label{rmk:g6-weak}
The conclusion of Proposition \ref{prop:g6} also holds for weak solutions in
the natural energy class. More precisely, let $\mu_1,\mu_2$ be two
non-negative weak solutions of
\[
    \partial_t\mu=L\Psi(\mu),
    \qquad
    \Psi(r)=r^{1+\beta},
\]
such that
\[
    \mu_i\in C([0,T],L^1(\R,\pi))\cap W^{1,1}_{\mathrm{loc}}(0,T;L^1_{\mathrm{loc}}(\mathbb R^d,\pi)),
    \qquad i=1,2,
\]
and
\[
    \Psi(\mu_i)\in L_{\mathrm{loc}}^2(0,T;H^1(\mathbb{R}^d,\pi)),\qquad i=1,2,
\]
Then, for every $0\leq t_1\leq t_2\leq T$, one has
\[
    \norm{(\mu_1(t_2)-\mu_2(t_2))^+}_{L^1(\R,\pi)}
    \leq
    \norm{(\mu_1(t_1)-\mu_2(t_1))^+}_{L^1(\R,\pi)}.
\]
Consequently,
\[
    \norm{\mu_1(t_2)-\mu_2(t_2)}_{L^1(\R,\pi)}
    \leq
    \norm{\mu_1(t_1)-\mu_2(t_1)}_{L^1(\R,\pi)}.
\]
The proof relies on a Kato-type inequality, which we postpone to Appendix~\ref{apx:weak12}.

\end{remark}
\begin{remark}\label{rmk:p11}
  If $\mu$ is a classical solution of \eqref{eq:p43} satisfying the assumptions of
Proposition \ref{prop:g6}, then applying Proposition \ref{prop:g6} with
$\mu_1=\mu$ and $\mu_2=0$ shows that
$\norm{\mu(t)}_{L^1(\R,\pi)}$ is non-increasing in time.
\end{remark}

\subsubsection{Aronson-Bénilan Estimate}
Under the additional convexity assumption on the potential $V$, we obtain an
Aronson--Bénilan type estimate analogous to the one for the classical porous
medium equation.
\begin{proposition}
Assume in addition that $V$ is convex, namely
\[
    \nabla^2V\geq0.
\]
Let $\mu$ be the positive classical solution from Proposition \ref{prop:re19},
and define
\[
    \nu:=\frac{\beta+1}{\beta}\mu^\beta.
\]
Then
\[
    L\nu\geq -\frac{1}{\beta t},
    \qquad \text{in } \R\times(0,T],
\]
where
\[
    L:=\Delta-\inner{\nabla V}{\nabla}.
\]
\end{proposition}

\begin{proof}
We prove the estimate first for smooth strictly positive approximating
solutions, for which all applications of the comparison principle are justified.

We begin by constructing smooth convex approximations of $V$ with bounded
derivatives. Let $V^*$ denote the Legendre transform of $V$, and define
\[
    W_i(x):=\sup_{\norm{p}\leq i}\{\inner{p}{x}-V^*(p)\}.
\]
Then $W_i$ is convex and globally Lipschitz, with
\[
    \norm{\nabla W_i}_{L^\infty(\R)}\leq i
\]
in the weak sense. Moreover, since $V$ is smooth and convex, for every compact
set $K\Subset\R$ and all sufficiently large $i$, we have
\[
    W_i=V
    \quad\text{on a neighbourhood of }K.
\]
Indeed, on such a compact set the quantity $\norm{\nabla V}$ is bounded, and
the identity
\[
    V(x)=\inner{\nabla V(x)}{x}-V^*(\nabla V(x))
\]
shows that the restricted supremum defining $W_i$ agrees with $V$ once
$i>\sup_K\norm{\nabla V}$.

Let $\rho_\varepsilon$ be a standard smooth mollifier and choose
$\varepsilon_i\downarrow0$. Set
\[
    V_i:=W_i*\rho_{\varepsilon_i}.
\]
Then
\[
    V_i\in C^\infty(\R),
    \qquad
    \nabla^2V_i\geq0,
\]
and $V_i$ has bounded derivatives of every order for each fixed $i$. Moreover,
\[
    V_i\to V
    \quad\text{in }C^\infty_{\mathrm{loc}}(\R).
\]
Define
\[
    L_i:=\Delta-\inner{\nabla V_i}{\nabla}.
\]

Let $S_i\in C^\infty(\R)$ be a cutoff such that
\[
    0\leq S_i\leq1,\qquad
    S_i=0\text{ on }B_i(0),\qquad
    S_i=1\text{ outside }B_{2i}(0).
\]
Set
\[
    \varsigma_i(x):=\varsigma(x)(1-S_i(x))+2^{-i}.
\]
Then $\varsigma_i$ is smooth, strictly positive, bounded, and equal to the
constant $2^{-i}$ outside a compact set. Let $\mu_i$ be the smooth solution of
\[
    \begin{aligned}
        &\partial_t\mu_i=L_i\mu_i^{\beta+1},
        \qquad \text{in }\R\times(0,T],\\
        &\mu_i(x,0)=\varsigma_i(x),
        \qquad x\in\R.
    \end{aligned}
\]
For each fixed $i$, this problem is uniformly parabolic, since the comparison
principle gives
\[
    2^{-i}\leq \mu_i(x,t)\leq \norm{\varsigma_i}_{L^\infty(\R)},
    \qquad (x,t)\in\R\times[0,T].
\]
Because $V_i$ has bounded derivatives of all orders and $\varsigma_i$ is smooth
with bounded derivatives, the solution $\mu_i$ is smooth and has bounded
derivatives on $\R\times[0,T]$ for each fixed $i$.

Define the pressure variable
\[
    \nu_i:=\frac{\beta+1}{\beta}\mu_i^\beta.
\]
By Lemma \ref{lem:re64}, applied with $L_i$ in place of $L$, we have
\[
    \partial_t\nu_i=\beta\nu_iL_i\nu_i+\norm{\nabla\nu_i}^2.
\]
Set
\[
    q_i:=L_i\nu_i.
\]
Since $L_i$ is independent of time,
\[
    \partial_tq_i
    =
    L_i\partial_t\nu_i
    =
    L_i\left(\beta\nu_iL_i\nu_i+\norm{\nabla\nu_i}^2\right).
\]
Using the product rule
\[
    L_i(\phi\psi)
    =
    \phi L_i\psi+\psi L_i\phi
    +
    2\inner{\nabla\phi}{\nabla\psi},
\]
we get
\[
    \begin{aligned}
        L_i(\nu_iL_i\nu_i)
        &=
        \nu_iL_i^2\nu_i
        +
        (L_i\nu_i)^2
        +
        2\inner{\nabla\nu_i}{\nabla L_i\nu_i}\\
        &=
        \nu_iL_iq_i
        +
        q_i^2
        +
        2\inner{\nabla\nu_i}{\nabla q_i}.
    \end{aligned}
\]
Therefore,
\[
    L_i(\beta\nu_iL_i\nu_i)
    =
    \beta\nu_iL_iq_i
    +
    \beta q_i^2
    +
    2\beta\inner{\nabla\nu_i}{\nabla q_i}.
\]

We compute the term $L_i\norm{\nabla\nu_i}^2$ using the weighted Bochner
identity. Since
\[
    L_i=\Delta-\inner{\nabla V_i}{\nabla},
\]
we have
\[
    L_i\norm{\nabla\nu_i}^2
    =
    2\inner{\nabla\nu_i}{\nabla L_i\nu_i}
    +
    2\norm{\nabla^2\nu_i}^2
    +
    2\nabla^2V_i(\nabla\nu_i,\nabla\nu_i).
\]
Recalling that $q_i=L_i\nu_i$, this identity can be rewritten as
\[
    L_i\norm{\nabla\nu_i}^2
    =
    2\inner{\nabla\nu_i}{\nabla q_i}
    +
    2\norm{\nabla^2\nu_i}^2
    +
    2\nabla^2V_i(\nabla\nu_i,\nabla\nu_i).
\]

Combining the previous identities, we obtain
\[
    \begin{aligned}
        \partial_tq_i
        &=
        \beta\nu_iL_iq_i
        +
        2(\beta+1)\inner{\nabla\nu_i}{\nabla q_i}
        +
        \beta q_i^2\\
        &\quad
        +
        2\norm{\nabla^2\nu_i}^2
        +
        2\nabla^2V_i(\nabla\nu_i,\nabla\nu_i).
    \end{aligned}
\]
Since $\nabla^2V_i\geq0$, the last two terms are non-negative. Thus
\[
    \partial_tq_i
    -
    \beta\nu_iL_iq_i
    -
    2(\beta+1)\inner{\nabla\nu_i}{\nabla q_i}
    -
    \beta q_i^2
    \geq0.
\]
Define
\[
    \mathcal{P}_ir
    :=
    \partial_tr
    -
    \beta\nu_iL_ir
    -
    2(\beta+1)\inner{\nabla\nu_i}{\nabla r}
    -
    \beta r^2.
\]
Then
\[
    \mathcal{P}_iq_i\geq0.
\]

For $\tau>0$, define
\[
    Q(t):=-\frac{1}{\beta(t+\tau)}.
\]
Since $Q$ is independent of $x$, we have
\[
    L_iQ=0,
    \qquad
    \nabla Q=0.
\]
Therefore,
\[
    \mathcal{P}_iQ
    =
    Q'(t)-\beta Q(t)^2.
\]
But
\[
    Q'(t)=\frac{1}{\beta(t+\tau)^2},
    \qquad
    \beta Q(t)^2=\frac{1}{\beta(t+\tau)^2}.
\]
Hence
\[
    \mathcal{P}_iQ=0.
\]

For each fixed $i$, the function $q_i(\cdot,0)=L_i\nu_i(\cdot,0)$ is bounded
from below. Indeed, $\varsigma_i$ is smooth, strictly positive, and constant
outside a compact set, while $V_i$ has bounded derivatives. Hence there exists
$M_i\geq0$ such that
\[
    q_i(x,0)\geq -M_i,
    \qquad \forall x\in\R.
\]
Since
\[
    Q(0)=-\frac{1}{\beta\tau}\to-\infty
    \qquad\text{as }\tau\downarrow0,
\]
we can choose $\tau>0$ sufficiently small so that
\[
    Q(0)\leq -M_i.
\]
Therefore,
\[
    q_i(x,0)\geq Q(0),
    \qquad \forall x\in\R.
\]

We now compare $q_i$ and $Q$. Set
\[
    z_i:=q_i-Q.
\]
Since $\mathcal{P}_iq_i\geq0$ and $\mathcal{P}_iQ=0$, we have
\[
    \begin{aligned}
        0
        &\leq
        \mathcal{P}_iq_i-\mathcal{P}_iQ\\
        &=
        \partial_tz_i
        -
        \beta\nu_iL_iz_i
        -
        2(\beta+1)\inner{\nabla\nu_i}{\nabla z_i}
        -
        \beta(q_i^2-Q^2)\\
        &=
        \partial_tz_i
        -
        \beta\nu_iL_iz_i
        -
        2(\beta+1)\inner{\nabla\nu_i}{\nabla z_i}
        -
        \beta(q_i+Q)z_i.
    \end{aligned}
\]
Thus $z_i$ satisfies the linear differential inequality
\[
    \partial_tz_i
    -
    \beta\nu_iL_iz_i
    -
    2(\beta+1)\inner{\nabla\nu_i}{\nabla z_i}
    -
    \beta(q_i+Q)z_i
    \geq0.
\]
For each fixed $i$, the coefficients are smooth and bounded on
$\R\times[0,T]$, and $\nu_i$ is bounded above and below away from zero. Hence
the operator is uniformly parabolic. Since
\[
    z_i(x,0)=q_i(x,0)-Q(0)\geq0,
\]
the parabolic comparison principle on $\R$ gives
\[
    z_i(x,t)\geq0,
    \qquad (x,t)\in\R\times[0,T].
\]
Equivalently,
\[
    q_i(x,t)\geq Q(t)
    =
    -\frac{1}{\beta(t+\tau)}.
\]
Letting $\tau\downarrow0$ gives
\[
    L_i\nu_i(x,t)=q_i(x,t)
    \geq
    -\frac{1}{\beta t},
    \qquad (x,t)\in\R\times(0,T].
\]

It remains to pass to the limit. By the local compactness argument used in
Proposition \ref{prop:re19}, after passing to a subsequence,
\[
    \mu_i\to\mu
    \quad\text{locally smoothly in }\R\times(0,T].
\]
Therefore,
\[
    \nu_i=\frac{\beta+1}{\beta}\mu_i^\beta
    \to
    \nu=\frac{\beta+1}{\beta}\mu^\beta
\]
locally smoothly in $\R\times(0,T]$. Since $V_i\to V$ in
$C^\infty_{\mathrm{loc}}(\R)$, we also have
\[
    L_i\nu_i
    =
    \Delta\nu_i-\inner{\nabla V_i}{\nabla\nu_i}
    \to
    \Delta\nu-\inner{\nabla V}{\nabla\nu}
    =
    L\nu
\]
locally uniformly in $\R\times(0,T]$. Passing to the limit in
\[
    L_i\nu_i\geq -\frac{1}{\beta t}
\]
yields
\[
    L\nu\geq -\frac{1}{\beta t}
    \qquad \text{in }\R\times(0,T].
\]
This proves the proposition.
\end{proof}

As a consequence of the preceding proposition, we have
\[
    L\nu\geq -\frac{1}{\beta t},
    \qquad
    \nu=\frac{\beta+1}{\beta}\mu^\beta .
\]
Since $L$ is linear, this implies
\[
    L\mu^\beta\geq -\frac{1}{(1+\beta)t}.
\]
Therefore, for every $t>0$, every $x_0\in\R$, and every non-negative
$\varphi\in H_0^1(B_1(x_0))$, integration by parts with respect to $d\pi$
gives
\[
    \int_{B_1(x_0)}
    \langle \nabla\mu^\beta(t),\nabla\varphi\rangle\,d\pi(x)
    \leq
    \int_{B_1(x_0)}
    \frac{1}{(1+\beta)t}\varphi\,d\pi(x).
\]
Applying Theorem \ref{thm:apx9} to $u=\mu^\beta(t)$ and
$g=(1+\beta)^{-1}t^{-1}$, we obtain, for every $r\geq1$,
\[
    \sup_{B_{1/2}(x_0)}\mu^\beta(t)
    \leq
    C\left(d,r,\max_{B_1(x_0)}V,\min_{B_1(x_0)}V\right)
    \left\{
    \|\mu^\beta(t)\|_{L^r(B_1(x_0),\pi)}
    +
    \frac{1}{(1+\beta)t}
    \right\}.
\]

We now distinguish two cases. If $0<\beta\leq1$, then $1/\beta\geq1$, and we
choose $r=1/\beta$. This gives
\[
    \|\mu^\beta(t)\|_{L^{1/\beta}(B_1(x_0),\pi)}
    =
    \left(\int_{B_1(x_0)}\mu(t)d\pi\right)^\beta
    \leq
    \|\mu(t)\|_{L^1(\R,\pi)}^\beta.
\]
Hence
\[
    \sup_{B_{1/2}(x_0)}\mu^\beta(t)
    \leq
    C\left(d,\beta,\max_{B_1(x_0)}V,\min_{B_1(x_0)}V\right)
    \left\{
    \|\mu(t)\|_{L^1(\R,\pi)}^\beta
    +
    \frac{1}{(1+\beta)t}
    \right\}.
\]

If $\beta>1$, we choose instead $r=1$. Then
\[
    \|\mu^\beta(t)\|_{L^1(B_1(x_0),\pi)}
    =
    \int_{B_1(x_0)}\mu^\beta(t)d\pi
    \leq
    \|\mu(t)\|_{L^\beta(\R,\pi)}^\beta.
\]
Using the $L^1$--$L^\beta$ smoothing estimate, we have, for every $t>0$,
\[
    \|\mu(t)\|_{L^\beta(\R,\pi)}
    \leq
    C\left(\beta,\lambda,t,\|\varsigma\|_{L^1(\R,\pi)}\right).
\]
Therefore,
\[
    \sup_{B_{1/2}(x_0)}\mu^\beta(t)
    \leq
    C\left(
    d,\beta,\lambda,t,\|\varsigma\|_{L^1(\R,\pi)},
    \max_{B_1(x_0)}V,\min_{B_1(x_0)}V
    \right).
\]

Combining the two cases, we obtain
\begin{equation}\label{eq:re110}
    \sup_{B_{1/2}(x_0)}\mu(t)
    \leq
    C\left(
    d,\beta,\lambda,t,\|\varsigma\|_{L^1(\R,\pi)},
    \max_{B_1(x_0)}V,\min_{B_1(x_0)}V
    \right)
    <\infty .
\end{equation}

\subsubsection{Cauchy Problem with Positive Initial Data in $L^1(\R,\pi)$}

In this subsection, we assume that the local upper bound \eqref{eq:re110} is
available. In particular, this holds under the convexity assumption on $V$ used
in the Aronson--B\'enilan estimate above. We also assume that
$\pi=e^{-V}\in\mathcal{P}(\R)$ satisfies the Poincar\'e inequality with
constant $\lambda>0$.
We proved the following proposition.

\begin{proposition}\label{prop:re23}
Assume that $V\in C^\infty(\R)$ is convex, that is,
\[
    \nabla^2 V\geq 0,
\]
and assume that $\pi=e^{-V}\in\mathcal{P}(\R)$ satisfies the Poincar\'e
inequality with constant $\lambda>0$. Let $\varsigma\in L^1(\R,\pi)$ be
locally strictly positive in the sense stated in the Introduction.

Then the Cauchy problem
\[
    \begin{aligned}
        &\partial_t\mu=L\mu^{\beta+1},
        \qquad L=\Delta-\inner{\nabla V}{\nabla},\\
        &\mu(x,0)=\varsigma(x),
    \end{aligned}
\]
admits a unique positive weak solution $\mu\in\Gamma'$, where
\[
    \begin{aligned}
        \Gamma':=\{\mu:\;&
        \mu\in\mathcal{C}([0,T],L^1(\R,\pi))
        \cap L^\infty(\tau,T;L^p(\R,\pi)),\quad
        \forall p\in(0,\infty),\ \forall \tau>0,\\
        &
        \partial_t\mu\in L^\infty(\tau,T;L^1(\R,\pi)),
        \quad \forall \tau>0,\\
        &
        \norm{\mu(t)}_{L^1(\R,\pi)}
        \text{ is non-increasing in }t,\\
        &
        \nabla\Psi(\mu)\in L^2(\tau,T;L^2(\R,\pi)),
        \quad
        L\Psi(\mu)\in L^\infty(\tau,T;L^1(\R,\pi)),
        \quad \forall \tau>0
        \}.
    \end{aligned}
\]
Moreover,
\[
    \mu\in C^\infty(\R\times(0,T]).
\]
In particular, the equation is satisfied classically on $\R\times(0,T]$, and
the initial condition is attained in $L^1(\R,\pi)$.
\end{proposition}
\begin{proof}
Let $\varsigma\in L^1(\R,\pi)$ be non-negative and locally strictly positive,
in the sense that for every compact set $Z\Subset\R$ there exists a constant
$c_Z>0$ such that
\[
    \varsigma(x)\geq c_Z
    \quad\text{for a.e. }x\in Z.
\]
Choose a sequence
\[
    \varsigma_i\in C^\infty(\R)
\]
such that
\[
    \varsigma_i>0,
    \qquad
    \varsigma_i\in L^\infty(\R),
    \qquad
    \varsigma_i\to\varsigma
    \quad\text{in }L^1(\R,\pi),
\]
and
\[
    \sup_i\norm{\varsigma_i}_{L^1(\R,\pi)}<\infty.
\]
Moreover, we choose the approximation so that for every compact set
$Z\Subset\R$, there exist constants $c_Z>0$ and $i_Z\in\mathbb{N}$ such that
\[
    \varsigma_i(x)\geq c_Z,
    \qquad \forall x\in Z,\quad \forall i\geq i_Z.
\]
Such an approximation can be obtained by truncating $\varsigma$, mollifying,
and adding a small positive constant.

For each $i$, Proposition \ref{prop:re19} gives a unique bounded positive
classical solution $\mu_i$ of
\[
    \begin{aligned}
        &\partial_t\mu_i=L\mu_i^{\beta+1},\\
        &\mu_i(x,0)=\varsigma_i(x),
    \end{aligned}
\]
where
\[
    L=\Delta-\inner{\nabla V}{\nabla}.
\]

We first prove local compactness. Let $Z\Subset\R$ and
$0<t_1<t_2\leq T$. By the local upper bound \eqref{eq:re110}, applied to
$\mu_i$, we have
\[
    \sup_{Z\times[t_1,t_2]}\mu_i
    \leq
    C(Z,t_1,t_2,\beta,V,\sup_i\norm{\varsigma_i}_{L^1(\R,\pi)}),
\]
where the constant is independent of $i$. On the other hand, by the local lower
bound in Lemma \ref{lem:re17}, using the uniform positivity of $\varsigma_i$ on
compact sets, there exists
\[
    c(Z,t_1,t_2)>0
\]
such that
\[
    \mu_i(x,t)\geq c(Z,t_1,t_2),
    \qquad
    \forall (x,t)\in Z\times[t_1,t_2],
\]
for all sufficiently large $i$. Thus the equation is uniformly parabolic on
$Z\times[t_1,t_2]$, with constants independent of $i$.

Consequently, by Theorem \ref{thm:apx7} and the Schauder estimate in Theorem
\ref{thm:apx8}, for every multi-index $k$ there exist $\delta\in(0,1)$ and a
constant $C$ independent of $i$ such that
\[
    \norm{D^k\mu_i}_{2+\delta,1+\delta/2;Z\times[t_1,t_2]}
    \leq C.
\]
By a diagonal argument, there exists a function $\mu$ such that, up to a
subsequence,
\[
    \mu_i\to\mu
    \quad\text{locally smoothly in }\R\times(0,T].
\]
Therefore,
\[
    \mu\in C^\infty(\R\times(0,T])
\]
and
\[
    \partial_t\mu=L\mu^{\beta+1}
    \quad\text{classically in }\R\times(0,T].
\]
Moreover, by the lower bound above,
\[
    \mu>0
    \quad\text{in }\R\times(0,T].
\]

Next we pass the a priori estimates to the limit. Since
\[
    \sup_i\norm{\varsigma_i}_{L^1(\R,\pi)}<\infty,
\]
the estimates from Proposition \ref{prop:re19-poincare} give, for every
$p\in(0,\infty)$ and every $\tau>0$,
\[
    \norm{\mu_i(t)}_{L^p(\R,\pi)}
    \leq
    C(p,\beta,\lambda,\tau,T,\sup_i\norm{\varsigma_i}_{L^1(\R,\pi)}),
    \qquad t\in[\tau,T],
\]
with a constant independent of $i$. Passing to the limit and applying Fatou's
lemma gives
\[
    \mu\in L^\infty(\tau,T;L^p(\R,\pi)),
    \qquad
    \forall p\in(0,\infty),\quad \forall \tau>0.
\]

Similarly, the energy estimate gives
\[
    \int_{\tau}^{T}\int_{\R}
    \norm{\nabla\Psi(\mu_i)}^2d\pi(x)dt
    \leq
    C(\tau,T,\beta,\lambda,\sup_i\norm{\varsigma_i}_{L^1(\R,\pi)}),
\]
where
\[
    \Psi(r):=r^{1+\beta}.
\]
Passing to the limit and applying Fatou's lemma yields
\[
    \int_{\tau}^{T}\int_{\R}
    \norm{\nabla\Psi(\mu)}^2d\pi(x)dt
    \leq
    C(\tau,T,\beta,\lambda,\sup_i\norm{\varsigma_i}_{L^1(\R,\pi)}).
\]
Hence
\[
    \nabla\Psi(\mu)\in L^2(\tau,T;L^2(\R,\pi)),
    \qquad \forall \tau>0.
\]

For the time derivative, Proposition \ref{prop:re19} gives
\[
    \norm{\partial_t\mu_i(t)}_{L^1(\R,\pi)}
    =
    \norm{L\Psi(\mu_i(t))}_{L^1(\R,\pi)}
    \leq
    \frac{2\norm{\varsigma_i}_{L^1(\R,\pi)}}{\beta t}
    \leq
    \frac{C}{t}.
\]
Since $\mu_i\to\mu$ locally smoothly in $\R\times(0,T]$, we have
\[
    \partial_t\mu_i\to\partial_t\mu
    \quad\text{pointwise in }\R\times(0,T].
\]
Therefore, by Fatou's lemma,
\[
    \norm{\partial_t\mu(t)}_{L^1(\R,\pi)}
    \leq
    \liminf_{i\to\infty}
    \norm{\partial_t\mu_i(t)}_{L^1(\R,\pi)}
    \leq
    \frac{C}{t}
\]
for every $t>0$. Since $\partial_t\mu=L\Psi(\mu)$ classically for positive
times, we also obtain
\[
    L\Psi(\mu)\in L^\infty(\tau,T;L^1(\R,\pi)),
    \qquad \forall \tau>0.
\]

It remains to prove that $\mu$ attains the initial datum in $L^1(\R,\pi)$ and
that $\mu\in C([0,T],L^1(\R,\pi))$.

By Proposition \ref{prop:g6}, for all $i,j$ and all $t\in[0,T]$,
\[
    \norm{\mu_i(t)-\mu_j(t)}_{L^1(\R,\pi)}
    \leq
    \norm{\varsigma_i-\varsigma_j}_{L^1(\R,\pi)}.
\]
We first pass to the limit in this estimate for positive times. Fix $i$ and
$t>0$. Since, along the chosen subsequence,
\[
    \mu_j(\cdot,t)\to \mu(\cdot,t)
    \quad\text{a.e. in }\R,
\]
Fatou's lemma gives
\[
    \begin{aligned}
        \norm{\mu_i(t)-\mu(t)}_{L^1(\R,\pi)}
        &=
        \int_{\R}|\mu_i(x,t)-\mu(x,t)|\,d\pi(x)\\
        &\leq
        \liminf_{j\to\infty}
        \int_{\R}|\mu_i(x,t)-\mu_j(x,t)|\,d\pi(x)\\
        &\leq
        \liminf_{j\to\infty}
        \norm{\varsigma_i-\varsigma_j}_{L^1(\R,\pi)}\\
        &=
        \norm{\varsigma_i-\varsigma}_{L^1(\R,\pi)}.
    \end{aligned}
\]
Thus, for every fixed $i$ and every $t>0$,
\[
    \norm{\mu_i(t)-\mu(t)}_{L^1(\R,\pi)}
    \leq
    \norm{\varsigma_i-\varsigma}_{L^1(\R,\pi)}.
\]

We now prove the initial trace. For every fixed $i$ and every $t>0$, by the
triangle inequality and the estimate above,
\[
    \begin{aligned}
        \norm{\mu(t)-\varsigma}_{L^1(\R,\pi)}
        &\leq
        \norm{\mu(t)-\mu_i(t)}_{L^1(\R,\pi)}
        +
        \norm{\mu_i(t)-\varsigma_i}_{L^1(\R,\pi)}
        +
        \norm{\varsigma_i-\varsigma}_{L^1(\R,\pi)}\\
        &\leq
        2\norm{\varsigma_i-\varsigma}_{L^1(\R,\pi)}
        +
        \norm{\mu_i(t)-\varsigma_i}_{L^1(\R,\pi)}.
    \end{aligned}
\]
For fixed $i$, Proposition \ref{prop:re19} gives
\[
    \mu_i\in C([0,T],L^1(\R,\pi)),
    \qquad
    \mu_i(0)=\varsigma_i.
\]
Hence
\[
    \norm{\mu_i(t)-\varsigma_i}_{L^1(\R,\pi)}
    \to0
    \qquad\text{as }t\downarrow0.
\]
Therefore,
\[
    \limsup_{t\downarrow0}
    \norm{\mu(t)-\varsigma}_{L^1(\R,\pi)}
    \leq
    2\norm{\varsigma_i-\varsigma}_{L^1(\R,\pi)}.
\]
Letting $i\to\infty$ and using
\[
    \varsigma_i\to\varsigma
    \quad\text{in }L^1(\R,\pi),
\]
we obtain
\[
    \lim_{t\downarrow0}
    \norm{\mu(t)-\varsigma}_{L^1(\R,\pi)}
    =
    0.
\]
Thus the initial datum is attained in $L^1(\R,\pi)$. We henceforth define
\[
    \mu(0):=\varsigma.
\]

It remains to prove continuity at positive times. Let $t\in(0,T)$ and let $h$
be such that $t+h\in(0,T)$. For fixed $i$, using the estimate above at the two
positive times $t+h$ and $t$, we get
\[
    \begin{aligned}
        \norm{\mu(t+h)-\mu(t)}_{L^1(\R,\pi)}
        &\leq
        \norm{\mu(t+h)-\mu_i(t+h)}_{L^1(\R,\pi)}\\
        &\quad+
        \norm{\mu_i(t+h)-\mu_i(t)}_{L^1(\R,\pi)}\\
        &\quad+
        \norm{\mu_i(t)-\mu(t)}_{L^1(\R,\pi)}\\
        &\leq
        2\norm{\varsigma_i-\varsigma}_{L^1(\R,\pi)}
        +
        \norm{\mu_i(t+h)-\mu_i(t)}_{L^1(\R,\pi)}.
    \end{aligned}
\]
For fixed $i$, since $\mu_i\in C([0,T],L^1(\R,\pi))$,
\[
    \norm{\mu_i(t+h)-\mu_i(t)}_{L^1(\R,\pi)}
    \to0
    \qquad\text{as }h\to0.
\]
Therefore,
\[
    \limsup_{h\to0}
    \norm{\mu(t+h)-\mu(t)}_{L^1(\R,\pi)}
    \leq
    2\norm{\varsigma_i-\varsigma}_{L^1(\R,\pi)}.
\]
Letting $i\to\infty$, we obtain
\[
    \lim_{h\to0}
    \norm{\mu(t+h)-\mu(t)}_{L^1(\R,\pi)}
    =
    0.
\]
Together with the continuity at $t=0$ proved above, this yields
\[
    \mu\in C([0,T],L^1(\R,\pi)).
\]

The $L^1$ norm of $\mu$ is non-increasing. Indeed, for any
$0<t_1<t_2\leq T$, the assumptions of Proposition \ref{prop:g6} are satisfied
on the interval $[t_1,t_2]$. Applying Proposition \ref{prop:g6} with
$\mu_1=\mu$ and $\mu_2=0$, we obtain
\[
    \norm{\mu(t_2)}_{L^1(\R,\pi)}
    \leq
    \norm{\mu(t_1)}_{L^1(\R,\pi)}.
\]
By the $L^1(\R,\pi)$-continuity of $\mu$ at $t=0$, this inequality also extends
to $t_1=0$.

We now prove uniqueness. Let $\mu_1$ and $\mu_2$ be two solutions in the same
class with the same initial datum $\varsigma$. For every $\tau>0$, we have
\[
    \int_\tau^T\int_{\R}
    \norm{\nabla\Psi(\mu_i)}^2d\pi dt<\infty,
    \qquad i=1,2.
\]
Hence Proposition \ref{prop:g6} can be applied on $[\tau,T]$. Therefore, for
every $t>\tau>0$,
\[
    \norm{\mu_1(t)-\mu_2(t)}_{L^1(\R,\pi)}
    \leq
    \norm{\mu_1(\tau)-\mu_2(\tau)}_{L^1(\R,\pi)}.
\]
Letting $\tau\downarrow0$ and using
$\mu_1,\mu_2\in C([0,T],L^1(\R,\pi))$, we get
\[
    \norm{\mu_1(t)-\mu_2(t)}_{L^1(\R,\pi)}
    \leq
    \norm{\mu_1(0)-\mu_2(0)}_{L^1(\R,\pi)}
    =
    0.
\]
Thus $\mu_1(t)=\mu_2(t)$ in $L^1(\R,\pi)$ for every $t\in[0,T]$, and the
solution is unique.
\end{proof}

\begin{remark}\label{rmk:re17}
    The comparison principle is preserved in the passage to the whole-space
    limit. Let $\varsigma_1,\varsigma_2\in C^\infty(\R)$ satisfy
    \[
        0<\varsigma_2\leq \varsigma_1\leq \frac{1}{\epsilon_2}.
    \]
    Let $\mu_1$ and $\mu_2$ be the corresponding solutions constructed in
    Proposition \ref{prop:re19}. We claim that
    \[
        \mu_1(x,t)\geq \mu_2(x,t),
        \qquad (x,t)\in\R\times[0,T].
    \]

    Indeed, choose the same cutoff functions for both initial data and set
    \[
        \varsigma_{k,i}:=\varsigma_k\chi_i,
        \qquad k=1,2.
    \]
    Then
    \[
        \varsigma_{1,i}\geq \varsigma_{2,i}
        \qquad\text{in }B_{R_i}(0).
    \]
    Let $\mu_{k,i}$ be the zero-boundary Cauchy--Dirichlet solution on
    $B_{R_i}(0)$ with initial datum $\varsigma_{k,i}$.

    Since $\mu_{k,i}$ may vanish on the boundary, the equation is degenerate
    there, and the uniformly parabolic comparison principle cannot be applied
    directly to $\mu_{1,i}$ and $\mu_{2,i}$. To justify the comparison, we first
    approximate the zero-boundary problems by strictly positive boundary
    problems. For $\delta>0$, let $\mu_{k,i}^{\delta}$ solve
    \[
        \begin{aligned}
            &\partial_t\mu_{k,i}^{\delta}
            =L\left(\mu_{k,i}^{\delta}\right)^{\beta+1},
            \qquad \text{in }B_{R_i}(0)\times(0,T],\\
            &\mu_{k,i}^{\delta}(x,0)=\varsigma_{k,i}(x)+\delta,
            \qquad x\in B_{R_i}(0),\\
            &\mu_{k,i}^{\delta}(x,t)=\delta,
            \qquad (x,t)\in\partial B_{R_i}(0)\times[0,T].
        \end{aligned}
    \]
    For each fixed $\delta>0$, these problems are uniformly parabolic, because
    the solutions stay bounded above and below away from zero. Moreover,
    \[
        \varsigma_{1,i}+\delta\geq \varsigma_{2,i}+\delta
    \]
    and the boundary data are equal. Hence the classical comparison principle
    gives
    \[
        \mu_{1,i}^{\delta}\geq \mu_{2,i}^{\delta}
        \qquad
        \text{in }B_{R_i}(0)\times[0,T].
    \]
    Passing to the limit $\delta\downarrow0$, and using the construction of the
    zero-boundary solutions in Proposition \ref{prop:re11}, we obtain
    \[
        \mu_{1,i}\geq \mu_{2,i}
        \qquad
        \text{in }B_{R_i}(0)\times[0,T].
    \]
    Finally, letting $i\to\infty$ in the whole-space approximation yields
    \[
        \mu_1(x,t)\geq \mu_2(x,t),
        \qquad (x,t)\in\R\times[0,T].
    \]

    The same argument applies to initial data satisfying the assumptions of
    Proposition \ref{prop:re23}. Such solutions are obtained as limits of
    bounded, positive, smooth initial data, and the ordering is preserved at
    each approximation level. Passing to the limit yields the comparison
    principle for the limiting solutions.
\end{remark}

\subsubsection{Cauchy Problem with Non-negative Initial Data in $L^1(\R,\pi)$}

We now remove the local strict positivity assumption on the initial datum. We proved the following proposition.

\begin{proposition}\label{prop:re24}
Assume that $V\in C^\infty(\R)$ is convex, that is,
\[
    \nabla^2V\geq0,
\]
and assume that $\pi=e^{-V}\in\mathcal{P}(\R)$ satisfies the Poincar\'e
inequality with constant $\lambda>0$. Then the following Cauchy problem has a
unique non-negative weak solution in the function class $\Gamma'$:
\begin{equation}\label{eq:g79}
    \begin{aligned}
        &\partial_t\mu=L\mu^{\beta+1},
        \qquad L=\Delta-\inner{\nabla V}{\nabla},\\
        &\mu(x,0)=\varsigma(x),
    \end{aligned}
\end{equation}
where $\varsigma\in L^1(\R,\pi)$ and $\varsigma\geq0$. Here $\Gamma'$ is the
class defined in Proposition \ref{prop:re23}. The equation is satisfied in the
weak sense on $\R\times(0,T)$, and the initial condition is attained in
$L^1(\R,\pi)$.
\end{proposition}

\begin{proof}
For each $i\in\mathbb{N}$, define
\[
    \varsigma_i:=\varsigma+2^{-i}.
\]
Then $\varsigma_i$ is locally strictly positive in the sense of Proposition
\ref{prop:re23}. Moreover, since $\pi(\R)=1$,
\[
    \norm{\varsigma_i-\varsigma}_{L^1(\R,\pi)}
    =
    2^{-i}\pi(\R)
    =
    2^{-i}\to0.
\]
Thus
\[
    \varsigma_i\to\varsigma
    \quad\text{in }L^1(\R,\pi).
\]
Also,
\[
    \norm{\varsigma_i}_{L^1(\R,\pi)}
    \leq
    \norm{\varsigma}_{L^1(\R,\pi)}+1.
\]

Let $\mu_i$ be the solution corresponding to the initial datum $\varsigma_i$,
as constructed in Proposition \ref{prop:re23}. Then
\[
    \mu_i\in\Gamma',
    \qquad
    \mu_i\in C^\infty(\R\times(0,T]).
\]
Since
\[
    \varsigma_i\geq \varsigma_{i+1},
\]
the comparison principle in Remark \ref{rmk:re17} gives
\[
    \mu_i(x,t)\geq \mu_{i+1}(x,t),
    \qquad (x,t)\in\R\times(0,T].
\]
Hence there exists a pointwise monotone limit
\[
    \mu_{\mathrm{pt}}(x,t):=\lim_{i\to\infty}\mu_i(x,t),
    \qquad (x,t)\in\R\times(0,T].
\]
Clearly,
\[
    \mu_{\mathrm{pt}}\geq0.
\]

We next prove convergence in $C([0,T],L^1(\R,\pi))$. Fix $i,j$ and let
$0<\tau<t\leq T$. Since $\mu_i,\mu_j\in\Gamma'$, we have
\[
    \int_\tau^T\int_{\R}\norm{\nabla\Psi(\mu_i)}^2d\pi ds<\infty,
    \qquad
    \int_\tau^T\int_{\R}\norm{\nabla\Psi(\mu_j)}^2d\pi ds<\infty.
\]
Therefore Proposition \ref{prop:g6} can be applied on the time interval
$[\tau,T]$. Hence
\[
    \norm{\mu_i(t)-\mu_j(t)}_{L^1(\R,\pi)}
    \leq
    \norm{\mu_i(\tau)-\mu_j(\tau)}_{L^1(\R,\pi)}.
\]
Since each $\mu_i$ is continuous from $[0,T]$ into $L^1(\R,\pi)$ and satisfies
\[
    \mu_i(0)=\varsigma_i,
    \qquad
    \mu_j(0)=\varsigma_j
\]
in $L^1(\R,\pi)$, we may let $\tau\downarrow0$ and obtain
\[
    \norm{\mu_i(t)-\mu_j(t)}_{L^1(\R,\pi)}
    \leq
    \norm{\varsigma_i-\varsigma_j}_{L^1(\R,\pi)},
    \qquad 0<t\leq T.
\]
For $t=0$, the same estimate is immediate:
\[
    \norm{\mu_i(0)-\mu_j(0)}_{L^1(\R,\pi)}
    =
    \norm{\varsigma_i-\varsigma_j}_{L^1(\R,\pi)}.
\]
Consequently,
\[
    \sup_{t\in[0,T]}
    \norm{\mu_i(t)-\mu_j(t)}_{L^1(\R,\pi)}
    \leq
    \norm{\varsigma_i-\varsigma_j}_{L^1(\R,\pi)}.
\]
Since $\varsigma_i\to\varsigma$ in $L^1(\R,\pi)$, the sequence
$\{\varsigma_i\}$ is Cauchy in $L^1(\R,\pi)$. Therefore $\{\mu_i\}$ is Cauchy
in the Banach space
\[
    C([0,T],L^1(\R,\pi)).
\]
Hence there exists a function
\[
    \mu\in C([0,T],L^1(\R,\pi))
\]
such that
\[
    \mu_i\to\mu
    \quad\text{in }C([0,T],L^1(\R,\pi)).
\]
Equivalently,
\[
    \sup_{t\in[0,T]}
    \norm{\mu_i(t)-\mu(t)}_{L^1(\R,\pi)}
    \to0.
\]

We now identify this $L^1$ limit with the pointwise monotone limit. For each
fixed $t>0$, the convergence
\[
    \mu_i(\cdot,t)\to\mu(\cdot,t)
    \quad\text{in }L^1(\R,\pi)
\]
implies that a subsequence converges to $\mu(\cdot,t)$ almost everywhere.
However, the full sequence $\mu_i(\cdot,t)$ converges pointwise monotonically
to $\mu_{\mathrm{pt}}(\cdot,t)$. Therefore the two limits agree:
\[
    \mu(\cdot,t)=\mu_{\mathrm{pt}}(\cdot,t)
    \quad\text{for a.e. }x\in\R.
\]
Thus, after modifying $\mu$ on a set of $\pi$-measure zero if necessary, we
identify $\mu$ with the pointwise monotone limit and simply write
\[
    \mu(x,t)=\lim_{i\to\infty}\mu_i(x,t),
    \qquad t>0.
\]

Evaluating the convergence in $C([0,T],L^1(\R,\pi))$ at $t=0$, we get
\[
    \norm{\mu(0)-\varsigma}_{L^1(\R,\pi)}
    \leq
    \norm{\mu(0)-\mu_i(0)}_{L^1(\R,\pi)}
    +
    \norm{\varsigma_i-\varsigma}_{L^1(\R,\pi)}.
\]
Since $\mu_i(0)=\varsigma_i$ and both terms on the right-hand side tend to
zero, we obtain
\[
    \mu(0)=\varsigma
    \quad\text{in }L^1(\R,\pi).
\]
Therefore the initial condition is attained in $L^1(\R,\pi)$.

We now pass the a priori estimates to the limit. Since
\[
    \sup_i\norm{\varsigma_i}_{L^1(\R,\pi)}<\infty,
\]
the estimates obtained in Proposition \ref{prop:re23} give, for every
$p\in(0,\infty)$ and every $\tau>0$,
\[
    \sup_{t\in[\tau,T]}
    \norm{\mu_i(t)}_{L^p(\R,\pi)}
    \leq
    C\left(p,\beta,\lambda,\tau,T,\norm{\varsigma}_{L^1(\R,\pi)}\right),
\]
where the constant is independent of $i$. By Fatou's lemma, for every
$t\in[\tau,T]$,
\[
    \norm{\mu(t)}_{L^p(\R,\pi)}
    \leq
    \liminf_{i\to\infty}
    \norm{\mu_i(t)}_{L^p(\R,\pi)}.
\]
Hence
\[
    \sup_{t\in[\tau,T]}
    \norm{\mu(t)}_{L^p(\R,\pi)}
    \leq
    C\left(p,\beta,\lambda,\tau,T,\norm{\varsigma}_{L^1(\R,\pi)}\right).
\]
Thus
\[
    \mu\in L^\infty(\tau,T;L^p(\R,\pi)),
    \qquad
    \forall p\in(0,\infty),\ \forall \tau>0.
\]

Similarly, the energy estimates for $\mu_i$ give
\[
    \int_{\tau}^{T}\int_{\R}
    \norm{\nabla\Psi(\mu_i)}^2d\pi(x)dt
    \leq
    C\left(\tau,T,\beta,\lambda,\norm{\varsigma}_{L^1(\R,\pi)}\right),
\]
where
\[
    \Psi(r):=r^{1+\beta}.
\]
Hence, after passing to a subsequence if necessary,
\[
    \nabla\Psi(\mu_i)\rightharpoonup \kappa
    \quad\text{weakly in }L^2(\tau,T;L^2(\R,\pi)).
\]
We claim that
\[
    \kappa=\nabla\Psi(\mu)
\]
in the sense of distributions. 

Let
\[
    E_\tau:=(\tau,T)\times\R .
\]
Since
\[
    \mu_i\to\mu
    \quad\text{in }C([0,T],L^1(\R,\pi)),
\]
we first have
\[
    \mu_i\to\mu
    \quad\text{strongly in }L^1(E_\tau,dtd\pi).
\]
Indeed,
\[
    \begin{aligned}
        \|\mu_i-\mu\|_{L^1(E_\tau)}
        &=
        \int_\tau^T
        \|\mu_i(t)-\mu(t)\|_{L^1(\R,\pi)}\,dt  \\
        &\leq
        (T-\tau)
        \sup_{t\in[0,T]}
        \|\mu_i(t)-\mu(t)\|_{L^1(\R,\pi)}
        \to0 .
    \end{aligned}
\]

We now upgrade this convergence to higher integrability. Fix any finite
$q>1$. Choose $s>q$. By the uniform estimates, we have
\[
    \sup_i\|\mu_i\|_{L^\infty(\tau,T;L^s(\R,\pi))}<\infty.
\]
Moreover, by Fatou's lemma,
\[
    \mu\in L^\infty(\tau,T;L^s(\R,\pi)),
\]
and hence
\[
    \sup_i\|\mu_i-\mu\|_{L^s(E_\tau)}<\infty.
\]
Interpolating between $L^1(E_\tau)$ and $L^s(E_\tau)$ gives
\[
    \|\mu_i-\mu\|_{L^q(E_\tau)}
    \leq
    \|\mu_i-\mu\|_{L^1(E_\tau)}^\theta
    \|\mu_i-\mu\|_{L^s(E_\tau)}^{1-\theta},
\]
where $\theta\in(0,1)$ is determined by
\[
    \frac1q=\theta+\frac{1-\theta}{s}.
\]
Since
\[
    \|\mu_i-\mu\|_{L^1(E_\tau)}\to0
\]
and
\[
    \sup_i\|\mu_i-\mu\|_{L^s(E_\tau)}<\infty,
\]
we obtain
\[
    \mu_i\to\mu
    \quad\text{strongly in }L^q(E_\tau)
\]
for every finite $q>1$. Together with the already proved $L^1$ convergence,
this gives
\[
    \mu_i\to\mu
    \quad\text{strongly in }L^q(E_\tau),
    \qquad \forall q\in[1,\infty).
\]

We now prove the strong convergence of the nonlinear term. Let
\[
    m:=1+\beta .
\]
Since $m<\infty$, the previous conclusion implies
\[
    \mu_i\to\mu
    \quad\text{strongly in }L^m(E_\tau).
\]
For non-negative numbers $a,b$, we have
\[
    |a^m-b^m|
    \leq
    m|a-b|\left(a^{m-1}+b^{m-1}\right).
\]
Applying this with $a=\mu_i$ and $b=\mu$, and then using H\"older's inequality,
we get
\[
    \begin{aligned}
        \|\mu_i^m-\mu^m\|_{L^1(E_\tau)}
        &\leq
        m\int_{E_\tau}
        |\mu_i-\mu|
        \left(\mu_i^{m-1}+\mu^{m-1}\right)\,d\pi dt \\
        &\leq
        m\|\mu_i-\mu\|_{L^m(E_\tau)}
        \left(
        \|\mu_i\|_{L^m(E_\tau)}^{m-1}
        +
        \|\mu\|_{L^m(E_\tau)}^{m-1}
        \right).
    \end{aligned}
\]
The factor
\[
    \left(
        \|\mu_i\|_{L^m(E_\tau)}^{m-1}
        +
        \|\mu\|_{L^m(E_\tau)}^{m-1}
    \right)
\]
is uniformly bounded, while
\[
    \|\mu_i-\mu\|_{L^m(E_\tau)}\to0.
\]
Therefore,
\[
    \mu_i^m\to\mu^m
    \quad\text{strongly in }L^1(E_\tau).
\]
Since $m=1+\beta$ and $\Psi(r)=r^{1+\beta}$, this gives
\[
    \Psi(\mu_i)=\mu_i^{1+\beta}
    \to
    \mu^{1+\beta}=\Psi(\mu)
    \quad\text{strongly in }L^1((\tau,T)\times\R,\pi).
\]
Therefore, for every test function
$\phi\in C_c^\infty(\R\times(\tau,T))$,
\[
    \int_{\tau}^{T}\int_{\R}
    \Psi(\mu_i)\,\partial_x\phi\,d\pi dt
    \to
    \int_{\tau}^{T}\int_{\R}
    \Psi(\mu)\,\partial_x\phi\,d\pi dt,
\]
while
\[
    \int_{\tau}^{T}\int_{\R}
    \nabla\Psi(\mu_i)\phi\,d\pi dt
    \to
    \int_{\tau}^{T}\int_{\R}
    \kappa\phi\,d\pi dt.
\]
Thus $\kappa=\nabla\Psi(\mu)$ in the distributional sense. Consequently,
\[
    \nabla\Psi(\mu)\in L^2(\tau,T;L^2(\R,\pi)),
    \qquad \forall\tau>0,
\]
and, by weak lower semicontinuity,
\[
    \int_{\tau}^{T}\int_{\R}
    \norm{\nabla\Psi(\mu)}^2d\pi(x)dt
    \leq
    C\left(\tau,T,\beta,\lambda,\norm{\varsigma}_{L^1(\R,\pi)}\right).
\]

We now verify that $\mu$ satisfies the equation weakly. For every
$0<t_1<t_2\leq T$ and every
$\phi\in C_c^\infty(\R\times[t_1,t_2])$, the smooth solution $\mu_i$ satisfies
\[
    \int_{t_1}^{t_2}\int_{\R}
    \left[
    -\mu_i\,\partial_t\phi
    +
    \inner{\nabla\Psi(\mu_i)}{\nabla\phi}
    \right]d\pi dt
    +
    \int_{\R}\mu_i(x,t)\phi(x,t)d\pi(x)\Big|_{t_1}^{t_2}
    =
    0.
\]
Using the strong convergence of $\mu_i$ in $C([0,T],L^1(\R,\pi))$ and the weak
convergence of $\nabla\Psi(\mu_i)$ in $L^2(t_1,t_2;L^2(\R,\pi))$, we may pass
to the limit and obtain
\[
    \int_{t_1}^{t_2}\int_{\R}
    \left[
    -\mu\,\partial_t\phi
    +
    \inner{\nabla\Psi(\mu)}{\nabla\phi}
    \right]d\pi dt
    +
    \int_{\R}\mu(x,t)\phi(x,t)d\pi(x)\Big|_{t_1}^{t_2}
    =
    0.
\]
Thus
\[
    \partial_t\mu=L\Psi(\mu)
\]
in the sense of distributions on $\R\times(0,T)$.

We also obtain a time-derivative estimate away from $t=0$. For each $i$,
\[
    \norm{\partial_t\mu_i(t)}_{L^1(\R,\pi)}
    =
    \norm{L\Psi(\mu_i(t))}_{L^1(\R,\pi)}
    \leq
    \frac{2\norm{\varsigma_i}_{L^1(\R,\pi)}}{\beta t}.
\]
Hence, for every $0<\tau<T$,
\[
    \sup_i\sup_{t\in[\tau,T]}
    \norm{\partial_t\mu_i(t)}_{L^1(\R,\pi)}
    \leq
    \frac{C}{\tau}.
\]
Equivalently,
\[
    \norm{\mu_i(t)-\mu_i(s)}_{L^1(\R,\pi)}
    \leq
    \frac{C}{\tau}|t-s|,
    \qquad s,t\in[\tau,T],
\]
with $C$ independent of $i$. Passing to the limit in
$C([0,T],L^1(\R,\pi))$, we get
\[
    \norm{\mu(t)-\mu(s)}_{L^1(\R,\pi)}
    \leq
    \frac{C}{\tau}|t-s|,
    \qquad s,t\in[\tau,T].
\]
Thus
\[
    \mu\in W^{1,\infty}(\tau,T;L^1(\R,\pi)).
\]
Since the distributional equation has already been identified, we conclude that
\[
    \partial_t\mu=L\Psi(\mu)
    \quad\text{in }L^\infty(\tau,T;L^1(\R,\pi)),
    \qquad \forall\tau>0.
\]
Equivalently,
\[
    L\Psi(\mu)\in L^\infty(\tau,T;L^1(\R,\pi)),
    \qquad \forall\tau>0.
\]

The $L^1$ norm of $\mu$ is non-increasing. Indeed, for any
$0<t_1<t_2\leq T$, the conditions of Remark \ref{rmk:g6-weak} are satisfied
on the interval $[t_1,t_2]$. Applying Remark~\ref{rmk:g6-weak} with
$\mu_1=\mu$ and $\mu_2=0$, we obtain
\[
    \norm{\mu(t_2)}_{L^1(\R,\pi)}
    \leq
    \norm{\mu(t_1)}_{L^1(\R,\pi)}.
\]
By the $L^1(\R,\pi)$-continuity of $\mu$ at $t=0$, this inequality also extends
to $t_1=0$.

Finally, we prove uniqueness. Let $\mu_1$ and $\mu_2$ be two solutions in the
same class with the same initial datum $\varsigma$. For every $\tau>0$, we have
\[
    \int_{\tau}^{T}\int_{\R}
    \norm{\nabla\Psi(\mu_i)}^2d\pi dt<\infty,
    \qquad i=1,2.
\]
Hence Remark~\ref{rmk:g6-weak} can be applied on $[\tau,T]$. Therefore, for
every $t>\tau>0$,
\[
    \norm{\mu_1(t)-\mu_2(t)}_{L^1(\R,\pi)}
    \leq
    \norm{\mu_1(\tau)-\mu_2(\tau)}_{L^1(\R,\pi)}.
\]
Letting $\tau\downarrow0$ and using
\[
    \mu_1,\mu_2\in C([0,T],L^1(\R,\pi)),
\]
we get
\[
    \norm{\mu_1(t)-\mu_2(t)}_{L^1(\R,\pi)}
    \leq
    \norm{\mu_1(0)-\mu_2(0)}_{L^1(\R,\pi)}
    =
    0.
\]
Thus $\mu_1(t)=\mu_2(t)$ in $L^1(\R,\pi)$ for every $t\in[0,T]$, and the
solution is unique.

\end{proof}

\subsection{Large Time Asymptotics}

We first show that the weak solution obtained above preserves its
$L^1(\R,\pi)$ mass. This improves the non-increasing property proved earlier.

\begin{lemma}[Mass conservation]
Let $\mu$ be a non-negative weak solution of
\[
    \partial_t\mu=L\Psi(\mu),
    \qquad
    \Psi(\mu)=\mu^{1+\beta},
\]
on $\R\times(0,T)$ such that
\[
    \mu\in C([0,T],L^1(\R,\pi)),
\]
and, for every $\tau>0$,
\[
    \nabla\Psi(\mu)\in L^2(\tau,T;L^2(\R,\pi)).
\]
Then
\[
    \norm{\mu(t)}_{L^1(\R,\pi)}
    =
    \norm{\mu(s)}_{L^1(\R,\pi)}
    \qquad
    \forall\,s,t\in[0,T].
\]
In particular, if $\mu(0)=\varsigma$ in $L^1(\R,\pi)$, then
\[
    \norm{\mu(t)}_{L^1(\R,\pi)}
    =
    \norm{\varsigma}_{L^1(\R,\pi)}
    \qquad
    \forall\,t\in[0,T].
\]
\end{lemma}

\begin{proof}
Let $\zeta_R\in C_c^\infty(\R)$ be a cutoff function such that
\[
    0\leq \zeta_R\leq1,
    \qquad
    \zeta_R=1 \text{ on } B_R(0),
    \qquad
    \zeta_R=0 \text{ outside } B_{2R}(0),
\]
and
\[
    \norm{\nabla\zeta_R}_{L^\infty(\R)}\leq \frac{C}{R}.
\]
Fix $0<s<t\leq T$. Using the weak formulation with the spatial test function
$\zeta_R$ and a standard approximation of the indicator function of $[s,t]$ in
time, we obtain
\[
    \int_{\R}\mu(x,t)\zeta_R(x)d\pi(x)
    -
    \int_{\R}\mu(x,s)\zeta_R(x)d\pi(x)
    =
    -\int_s^t\int_{\R}
    \langle\nabla\Psi(\mu),\nabla\zeta_R\rangle d\pi d\tau .
\]
Therefore,
\[
\begin{aligned}
    &\left|
    \int_{\R}\mu(x,t)\zeta_R(x)d\pi(x)
    -
    \int_{\R}\mu(x,s)\zeta_R(x)d\pi(x)
    \right| \\
    &\qquad\leq
    \left(
    \int_s^t\int_{\R}
    \norm{\nabla\Psi(\mu)}^2d\pi d\tau
    \right)^{1/2}
    \left(
    \int_s^t\int_{\R}
    \norm{\nabla\zeta_R}^2d\pi d\tau
    \right)^{1/2}.
\end{aligned}
\]
Since $\nabla\zeta_R$ is supported in $B_{2R}(0)\setminus B_R(0)$, we have
\[
    \int_s^t\int_{\R}
    \norm{\nabla\zeta_R}^2d\pi d\tau
    \leq
    \frac{C(t-s)}{R^2}
    \pi\left(B_{2R}(0)\setminus B_R(0)\right).
\]
Because $\pi$ is a probability measure,
\[
    \pi\left(B_{2R}(0)\setminus B_R(0)\right)\to0
    \qquad\text{as }R\to\infty.
\]
Hence
\[
    \int_s^t\int_{\R}
    \norm{\nabla\zeta_R}^2d\pi d\tau
    \to0
    \qquad\text{as }R\to\infty.
\]
It follows that
\[
    \lim_{R\to\infty}
    \left|
    \int_{\R}\mu(x,t)\zeta_R(x)d\pi(x)
    -
    \int_{\R}\mu(x,s)\zeta_R(x)d\pi(x)
    \right|
    =
    0.
\]
Since $\mu(\cdot,t),\mu(\cdot,s)\in L^1(\R,\pi)$ and
$\zeta_R\uparrow1$, the monotone convergence theorem gives
\[
    \int_{\R}\mu(x,t)d\pi(x)
    =
    \int_{\R}\mu(x,s)d\pi(x),
    \qquad 0<s<t\leq T.
\]
Finally, since $\mu\in C([0,T],L^1(\R,\pi))$, letting $s\downarrow0$ yields
\[
    \norm{\mu(t)}_{L^1(\R,\pi)}
    =
    \norm{\mu(0)}_{L^1(\R,\pi)}
    =
    \norm{\varsigma}_{L^1(\R,\pi)}.
\]
This proves the lemma.
\end{proof}

Therefore, the solution constructed in Proposition \ref{prop:re24} belongs to
the class $\Gamma$ appearing in Theorem \ref{thm:f1}; namely, its $L^1$ mass is
conserved rather than merely non-increasing.

We now prove the large-time estimate.

\begin{proposition}
	\label{prop:refinal}
Assume that $\pi=e^{-V}\in\mathcal{P}(\R)$ satisfies the Poincar\'e inequality
with constant $\lambda>0$. Let $\mu\in\Gamma$ be the non-negative weak solution
of
\[
    \begin{aligned}
        &\partial_t\mu
        =
        \frac{1}{1+\beta}
        \left(\Delta\mu^{1+\beta}
        -
        \langle\nabla V,\nabla\mu^{1+\beta}\rangle\right),\\
        &\mu(x,0)=\varsigma(x),
        \qquad
        \varsigma\in L^1(\R,\pi),\quad \varsigma\geq0.
    \end{aligned}
\]
Let
\[
    M:=\norm{\varsigma}_{L^1(\R,\pi)}.
\]
If $M=0$, then $\mu\equiv0$. If $M>0$, then for every
$p>1$, $\beta>0$, and $p+\beta\geq2$, we have
\[
    \Renyi(\mu(t))
    \leq
    \begin{cases}
    -\dfrac{
    \log\left(
    e^{-\beta \Renyi(\varsigma)}
    +
    \dfrac{2\beta\lambda(p-1)}{(p+\beta)^2}
    M^\beta t
    \right)}
    {\beta},
    & 0<t<t_0,\\[2ex]
    e^{-\frac{2p\lambda M^\beta(t-t_0)}{(p+\beta)^2}}
    \Renyi(\mu(t_0)),
    & t\geq t_0,
    \end{cases}
\]
where
\[
    \Renyi(\mu(t))
    :=
    \log\left(\frac{\norm{\mu(t)}_{L^p(\R,\pi)}}{M}\right)
    \geq0,
\]
and
\[
    t_0:=
    \inf\left\{
    t\geq0:
    \Renyi(\mu(t))\leq \frac{p-1}{p}
    \right\}.
\]
If $\varsigma\notin L^p(\R,\pi)$, then we use the convention
\[
    \Renyi(\varsigma)=+\infty,
    \qquad
    e^{-\beta\Renyi(\varsigma)}=0.
\]
\end{proposition}

\begin{proof}
If $M=0$, then $\varsigma=0$ in $L^1(\R,\pi)$. By uniqueness, $\mu\equiv0$.
Thus we assume $M>0$.

Define the normalized solution
\[
    \bar{\mu}(x,t):=
    \frac{1}{M}\mu\left(x,M^{-\beta}t\right).
\]
Then $\bar{\mu}$ satisfies
\[
    \partial_t\bar{\mu}
    =
    \frac{1}{1+\beta}
    \left(\Delta\bar{\mu}^{1+\beta}
    -
    \langle\nabla V,\nabla\bar{\mu}^{1+\beta}\rangle\right),
\]
with initial datum
\[
    \bar{\mu}(x,0)=\frac{\varsigma(x)}{M}.
\]
By the mass conservation lemma,
\[
    \norm{\bar{\mu}(t)}_{L^1(\R,\pi)}=1,
    \qquad t\geq0.
\]

We first prove the estimate for the normalized solution. Applying Theorem
\ref{thm:4} to the normalized Cauchy--Dirichlet approximations used in the
construction of $\bar{\mu}$, and then passing to the limit by Fatou's lemma,
gives
\[
    K_p(\bar{\mu}(t))
    \leq
    \begin{cases}
    -\dfrac{p}{
    \beta(p-1)}
    \log\left(
    e^{-\frac{\beta(p-1)}{p}K_p(\bar{\mu}(0))}
    +
    \dfrac{2\beta\lambda(p-1)}{(p+\beta)^2}t
    \right),
    &0<t<\bar{t}_0,\\[2ex]
    e^{-\frac{2p\lambda(t-\bar{t}_0)}{(p+\beta)^2}}
    K_p(\bar{\mu}(\bar{t}_0)),
    &t\geq \bar{t}_0,
    \end{cases}
\]
where
\[
    K_p(f):=\frac{1}{p-1}\log\left(\int_{\R}f^p\,d\pi\right),
\]
and
\[
    \bar{t}_0:=
    \inf\left\{
    t\geq0:K_p(\bar{\mu}(t))\leq1
    \right\}.
\]
The passage to the limit is justified as follows. The approximating solutions
converge to $\bar{\mu}$ in $L^1_{\mathrm{loc}}((0,T];L^1(\R,\pi))$ and
almost everywhere along a subsequence. Hence Fatou's lemma yields
\[
    \int_{\R}\bar{\mu}^p(t)d\pi
    \leq
    \liminf_i
    \int_{\R}\bar{\mu}_i^p(t)d\pi
\]
for every $t>0$. Thus the upper bounds for the approximations pass to the
limit. For the initial datum, the approximations may be chosen so that their
$L^p$ norms converge monotonically to the $L^p$ norm of
$\varsigma/M$, with the value $+\infty$ allowed. Therefore the convention
$e^{-\beta\Renyi(\varsigma)}=0$ is consistent when
$\varsigma\notin L^p(\R,\pi)$.

Now observe that
\[
    \int_{\R}\bar{\mu}^p(t)d\pi
    =
    \frac{1}{M^p}
    \int_{\R}\mu^p(M^{-\beta}t)d\pi.
\]
Hence
\[
    K_p(\bar{\mu}(t))
    =
    \frac{1}{p-1}
    \log\left(
    \frac{\norm{\mu(M^{-\beta}t)}_{L^p(\R,\pi)}^p}{M^p}
    \right)
    =
    \frac{p}{p-1}
    \Renyi\left(\mu(M^{-\beta}t)\right).
\]
Equivalently,
\[
    \Renyi(\mu(s))
    =
    \frac{p-1}{p}K_p(\bar{\mu}(M^\beta s)).
\]
The condition
\[
    K_p(\bar{\mu}(M^\beta s))\leq1
\]
is equivalent to
\[
    \Renyi(\mu(s))\leq\frac{p-1}{p}.
\]
Thus
\[
    \bar{t}_0=M^\beta t_0.
\]

Substituting $t=M^\beta s$ in the estimate for $K_p(\bar{\mu}(t))$, and then
multiplying by $(p-1)/p$, gives
\[
    \Renyi(\mu(s))
    \leq
    -\frac{1}{\beta}
    \log\left(
    e^{-\beta\Renyi(\varsigma)}
    +
    \frac{2\beta\lambda(p-1)}{(p+\beta)^2}
    M^\beta s
    \right)
\]
for $0<s<t_0$, and
\[
    \Renyi(\mu(s))
    \leq
    e^{-\frac{2p\lambda M^\beta(s-t_0)}{(p+\beta)^2}}
    \Renyi(\mu(t_0))
\]
for $s\geq t_0$. Renaming $s$ as $t$ proves the proposition.
\end{proof}

\section{Conclusion}

In this paper, we studied a weighted porous medium equation associated with a
Gibbs probability measure $\pi=e^{-V}$ on $\R$. Under the assumptions that
$V$ is smooth and convex and that $\pi$ satisfies a Poincar\'e inequality, we
proved the well-posedness of the Cauchy problem for non-negative initial data in
$L^1(\R,\pi)$. The solution is constructed by first solving suitable
Cauchy--Dirichlet problems on bounded domains and then passing to the
whole-space limit. The $L^1$ contraction principle plays a central role in
establishing uniqueness, stability, and convergence of the approximating
solutions.

A key difficulty is that, unlike the classical porous medium equation, the
presence of the drift term generated by the potential $V$ prevents the direct
use of standard Barenblatt-type barriers. To overcome this, we introduced a
local lower barrier for the pressure variable and combined it with an
Aronson--B\'enilan type estimate. This yields local upper and lower bounds at
positive times, which are sufficient to obtain local uniform parabolicity and
hence higher regularity for locally strictly positive data.

We also established mass conservation for weak solutions and proved an
$L^1$--$L^p$ smoothing effect. More precisely, for every admissible
$p>1$, solutions with merely $L^1(\R,\pi)$ initial data become
$L^p(\R,\pi)$ at every positive time. Moreover, the logarithmic ratio between
the $L^p(\R,\pi)$ norm of the solution and its conserved $L^1(\R,\pi)$ mass
satisfies a two-phase decay estimate: it first decays at a super-exponential
rate and then decays exponentially to zero. This shows convergence, in the
corresponding $L^p$ sense, toward the constant equilibrium determined by the
initial mass.

Several natural questions remain open. One possible direction is to weaken the
convexity assumption on $V$, which is used in the Aronson--B\'enilan estimate.
Another direction is to investigate whether analogous smoothing and decay
estimates hold under functional inequalities weaker than the Poincar\'e
inequality, or for more general weighted nonlinear diffusion equations. These
questions will be addressed in future work.

\section*{Acknowledgements and Competing Interests}

%

The author acknowledges the support of the Munich Center for Machine Learning and appreciates the valuable suggestions provided by Prof. Massimo Fornasier.\\

\noindent There is no conflicts of interest.


\bibliographystyle{apalike}
\bibliography{bibliography}

\section{Appendix}

\subsection{Weak Version of Proposition \ref{prop:g6}}\label{apx:weak12}
\begin{proposition}[$L^1$ contraction principle for weak solutions]\label{prop:g6-weak}
Let
\[
    L:=\Delta-\langle\nabla V,\nabla\rangle,
    \qquad
    \Psi(r):=r^{1+\beta}.
\]
Assume that \(\pi=e^{-V}\) is a probability measure on
\(\mathbb R^d\). Let \(\mu_1,\mu_2\) be two nonnegative weak solutions of
\[
    \partial_t\mu=L\Psi(\mu)
\]
on \(\mathbb R^d\times(0,T)\), satisfying
\[
    \mu_i\in C([0,T],L^1(\mathbb R^d,\pi))\cap W^{1,1}_{\mathrm{loc}}(0,T;L^1_{\mathrm{loc}}(\mathbb R^d,\pi)),
    \qquad i=1,2,
\]
and, for every \(\tau>0\),
\[
    \Psi(\mu_i)\in
    L^2_{\mathrm{loc}}(0,T;H^1_{\mathrm{loc}}(\mathbb R^d,\pi)).
\]

Assume moreover that, for every \(0<s_1<s_2\le T\) and every
\[
    \phi\in C^1([s_1,s_2];C_c^\infty(\mathbb R^d)),
\]
one has
\[
    \int_{s_1}^{s_2}\int_{\mathbb R^d}
    \left[
    -\mu_i\,\partial_t\phi
    +
    \langle\nabla\Psi(\mu_i),\nabla\phi\rangle
    \right]d\pi dt
    +
    \int_{\mathbb R^d}\mu_i(x,t)\phi(x,t)d\pi(x)\Big|_{s_1}^{s_2}
    =
    0.
\]
Then, for every
\[
    0\leq t_1\le t_2\le T,
\]
we have
\[
    \|(\mu_1(t_2)-\mu_2(t_2))^+\|_{L^1(\mathbb R^d,\pi)}
    \le
    \|(\mu_1(t_1)-\mu_2(t_1))^+\|_{L^1(\mathbb R^d,\pi)}.
\]
Consequently,
\[
    \|\mu_1(t_2)-\mu_2(t_2)\|_{L^1(\mathbb R^d,\pi)}
    \le
    \|\mu_1(t_1)-\mu_2(t_1)\|_{L^1(\mathbb R^d,\pi)}.
\]

\end{proposition}

\begin{proof}
Set
\[
    y:=\mu_1-\mu_2,
    \qquad
    w:=\Psi(\mu_1)-\Psi(\mu_2).
\]
Since \(\Psi(r)=r^{1+\beta}\) is strictly increasing on \([0,\infty)\),
\[
    yw\ge0
    \qquad\text{a.e.}
\]
Subtracting the two weak formulations gives
\[
    \partial_t y=Lw
\]
in the sense of distributions on \(\mathbb R^d\times(0,T)\).

Fix
\[
    0<\tau\le t_1<t_2\le T,
\]
 Lemma~\ref{lem:weak-kato} gives
\[
    \partial_t y^+\le Lw^+
\]
in the sense of distributions on \(\mathbb R^d\times(t_1,t_2)\).

Let
\[
    \zeta\in C_c^\infty(\mathbb R^d),
    \qquad
    \zeta\ge0,
\]
\[
Y(t):=\int_{\mathbb R^d}y^+(x,t)\zeta(x)\,d\pi(x)
\]
and
\[
B(t):=
-\int_{\mathbb R^d}
\langle \nabla w^+(x,t),\nabla\zeta(x)\rangle\,d\pi(x).
\]
Since \(y^+\in C([t_1,t_2];L^1(\mathbb R^d,\pi))\), we have
\(Y\in C([t_1,t_2])\). Moreover \(B\in L^1(t_1,t_2)\).

Let \(\eta_\varepsilon\in C_c^\infty((a,b))\) be a nonnegative cutoff satisfying
\[
0\le \eta_\varepsilon\le1,
\qquad
\eta_\varepsilon\to \mathbf 1_{(t_1,t_2)}
\quad\text{a.e.},
\]
and such that \(\eta_\varepsilon\) increases from \(0\) to \(1\) on
\((t_1,t_1+\varepsilon)\) and decreases from \(1\) to \(0\) on
\((t_2-\varepsilon,t_2)\). Testing
\[
\partial_t y^+\le Lw^+
\]
against
\[
\varphi(x,t)=\eta_\varepsilon(t)\zeta(x)
\]
gives
\[
-\int_a^b Y(t)\eta_\varepsilon'(t)\,dt
\le
\int_a^b \eta_\varepsilon(t)B(t)\,dt.
\]
Letting \(\varepsilon\downarrow0\), using the continuity of \(Y\) and dominated
convergence for \(B\), yields
\[
Y(t_2)-Y(t_1)
\le
\int_{t_1}^{t_2}B(t)\,dt.
\]
Therefore
\[
\int_{\mathbb R^d}y^+(x,t_2)\zeta(x)\,d\pi(x)
-
\int_{\mathbb R^d}y^+(x,t_1)\zeta(x)\,d\pi(x)
\le
-\int_{t_1}^{t_2}\int_{\mathbb R^d}
\langle\nabla w^+,\nabla\zeta\rangle\,d\pi dt.
\]
Using
\[
\nabla w^+
=
\operatorname{sign}_0^+(w)\nabla w
\]
gives
\[
\int_{\mathbb R^d}y^+(x,t_2)\zeta(x)\,d\pi(x)
-
\int_{\mathbb R^d}y^+(x,t_1)\zeta(x)\,d\pi(x)
\le
-\int_{t_1}^{t_2}\int_{\mathbb R^d}
\operatorname{sign}_0^+(w)
\langle\nabla w,\nabla\zeta\rangle\,d\pi dt.
\]

We now remove the spatial cutoff. Choose
\(\zeta_1\in C_c^\infty(\mathbb R^d)\) such that
\[
    0\le \zeta_1\le1,
    \qquad
    \zeta_1(x)=1 \quad\text{for } |x|\le1,
    \qquad
    \zeta_1(x)=0 \quad\text{for } |x|\ge2.
\]
Define
\[
    \zeta_n(x):=\zeta_1(x/n).
\]
Then
\[
    0\le \zeta_n\le1,
    \qquad
    \zeta_n(x)\to1,
    \qquad
    \|\nabla\zeta_n\|_{L^\infty(\mathbb R^d)}
    \le \frac{C}{n}.
\]
Applying the localized inequality with \(\zeta=\zeta_n\), we get
\[
    \begin{aligned}
        &\int_{\mathbb R^d}y^+(x,t_2)\zeta_n(x)\,d\pi(x)
        -
        \int_{\mathbb R^d}y^+(x,t_1)\zeta_n(x)\,d\pi(x)\\
        &\qquad\le
        \left|
        \int_{t_1}^{t_2}\int_{\mathbb R^d}
        \mathbf 1_{\{w>0\}}
        \langle\nabla w,\nabla\zeta_n\rangle\,d\pi dt
        \right|.
    \end{aligned}
\]
By Cauchy--Schwarz,
\[
    \begin{aligned}
        &\left|
        \int_{t_1}^{t_2}\int_{\mathbb R^d}
        \mathbf 1_{\{w>0\}}
        \langle\nabla w,\nabla\zeta_n\rangle\,d\pi dt
        \right|\\
        &\qquad\le
        \left(
        \int_{t_1}^{t_2}\int_{\mathbb R^d}
        |\nabla w|^2\,d\pi dt
        \right)^{1/2}
        \left(
        \int_{t_1}^{t_2}\int_{\mathbb R^d}
        |\nabla\zeta_n|^2\,d\pi dt
        \right)^{1/2}.
    \end{aligned}
\]
The first factor is finite by the energy assumption. Since \(\pi\) is a
probability measure,
\[
    \int_{t_1}^{t_2}\int_{\mathbb R^d}
    |\nabla\zeta_n|^2\,d\pi dt
    \le
    \frac{C(t_2-t_1)}{n^2}.
\]
Hence the cutoff error tends to \(0\) as \(n\to\infty\). Therefore,
\[
    \int_{\mathbb R^d}y^+(x,t_2)\zeta_n(x)\,d\pi(x)
    \le
    \int_{\mathbb R^d}y^+(x,t_1)\zeta_n(x)\,d\pi(x)
    +o(1).
\]
Letting \(n\to\infty\), and using dominated convergence, we obtain
\[
    \int_{\mathbb R^d}y^+(x,t_2)\,d\pi(x)
    \le
    \int_{\mathbb R^d}y^+(x,t_1)\,d\pi(x).
\]
That is,
\[
    \|(\mu_1(t_2)-\mu_2(t_2))^+\|_{L^1(\mathbb R^d,\pi)}
    \le
    \|(\mu_1(t_1)-\mu_2(t_1))^+\|_{L^1(\mathbb R^d,\pi)}.
\]

Exchanging the roles of \(\mu_1\) and \(\mu_2\), we also obtain
\[
    \|(\mu_2(t_2)-\mu_1(t_2))^+\|_{L^1(\mathbb R^d,\pi)}
    \le
    \|(\mu_2(t_1)-\mu_1(t_1))^+\|_{L^1(\mathbb R^d,\pi)}.
\]
Adding the two inequalities yields
\[
    \|\mu_1(t_2)-\mu_2(t_2)\|_{L^1(\mathbb R^d,\pi)}
    \le
    \|\mu_1(t_1)-\mu_2(t_1)\|_{L^1(\mathbb R^d,\pi)}.
\]

For $t_1=0$, we use the assumption $\mu_i\in C([0,T],L^1(\mathbb{R}^d,\pi)).$

\end{proof}

\begin{lemma}[Weak Kato inequality for monotone pairs]\label{lem:weak-kato}
Assume \[
    y=\mu_1-\mu_2,
    \qquad
    w=\Psi(\mu_1)-\Psi(\mu_2),
    \qquad
    \Psi(r)=r^{1+\beta},
\]
with \(\mu_1,\mu_2\ge0\) and 
\[
    \mu_i\in W^{1,1}_{\mathrm{loc}}(0,T;L^1_{\mathrm{loc}}(\mathbb R^d,\pi)),
    \qquad
    \Psi(\mu_i)\in L^2_{\mathrm{loc}}(0,T;H^1_{\mathrm{loc}}(\mathbb R^d,\pi)),
\]
Suppose that
\[
    \partial_t y=Lw
\]
in the weak sense, namely
\[
    -\int_0^T\int_{\mathbb R^d}  y\,\partial_t\varphi\,d\pi dt
    =
    -\int_0^T\int_{\mathbb R^d}
    \langle \nabla w,\nabla\varphi\rangle\,d\pi dt
\]
for every \(\varphi\in C_c^\infty(\mathbb R^d\times (0,T))\).
 Then
\[
    \partial_t y^+\le Lw^+
\]
in the sense of distributions. Equivalently, for every
\[
    \varphi\in C_c^\infty(\mathbb R^d\times (0,T)),
    \qquad
    \varphi\ge0,
\]
one has
\[
    -\int_0^T\int_{\mathbb R^d} y^+\,\partial_t\varphi\,d\pi dt
    \le
    -\int_0^T\int_{\mathbb R^d}
    \langle \nabla w^+,\nabla\varphi\rangle\,d\pi dt.
\]
\end{lemma}

\begin{proof}
Set
\[
    \chi_y:=\mathbf 1_{\{y>0\}},
    \qquad
    \chi_w:=\mathbf 1_{\{w>0\}}.
\]
Since
\[
    y=\mu_1-\mu_2,
    \qquad
    w=\Psi(\mu_1)-\Psi(\mu_2),
    \qquad
    \Psi(r)=r^{1+\beta},
\]
and since \(\Psi\) is strictly increasing on \([0,\infty)\), we have
\[
    \mu_1>\mu_2
    \iff
    \Psi(\mu_1)>\Psi(\mu_2),
\]
\[
    \mu_1=\mu_2
    \iff
    \Psi(\mu_1)=\Psi(\mu_2),
\]
and
\[
    \mu_1<\mu_2
    \iff
    \Psi(\mu_1)<\Psi(\mu_2).
\]
Therefore
\[
    y>0\iff w>0,
    \qquad
    y=0\iff w=0,
    \qquad
    y<0\iff w<0.
\]
Hence
\[
    \chi_y=\chi_w
    \qquad\text{a.e. on }\mathbb R^d\times(0,T).
\]
We denote this common function by
\[
    \chi:=\chi_y=\chi_w.
\]

Since
\[
    \mu_i\in W^{1,1}_{\mathrm{loc}}
    (0,T;L^1_{\mathrm{loc}}(\mathbb R^d,\pi)),
\]
we have
\[
    y=\mu_1-\mu_2
    \in W^{1,1}_{\mathrm{loc}}
    (0,T;L^1_{\mathrm{loc}}(\mathbb R^d,\pi)).
\]
Because the map \(r\mapsto r^+\) is Lipschitz, the chain rule gives
\[
    y^+\in W^{1,1}_{\mathrm{loc}}
    (0,T;L^1_{\mathrm{loc}}(\mathbb R^d,\pi)),
\]
and
\[
    \partial_t y^+
    =
    \mathbf 1_{\{y>0\}}\partial_t y
    =
    \chi\,\partial_t y
    \qquad\text{a.e. on }\mathbb R^d\times(0,T).
\]

Similarly, since
\[
    \Psi(\mu_i)\in L^2_{\mathrm{loc}}
    (0,T;H^1_{\mathrm{loc}}(\mathbb R^d,\pi)),
\]
we have
\[
    w\in L^2_{\mathrm{loc}}
    (0,T;H^1_{\mathrm{loc}}(\mathbb R^d,\pi)).
\]
Again by the Sobolev chain rule,
\[
    w^+\in L^2_{\mathrm{loc}}
    (0,T;H^1_{\mathrm{loc}}(\mathbb R^d,\pi)),
\]
and
\[
    \nabla w^+
    =
    \mathbf 1_{\{w>0\}}\nabla w
    =
    \chi\nabla w
    \qquad\text{a.e. on }\mathbb R^d\times(0,T).
\]

Now let
\[
    \varphi\in C_c^\infty(\mathbb R^d\times(0,T)),
    \qquad
    \varphi\ge0.
\]
By assumption, for every
\[
    \zeta\in C_c^\infty(\mathbb R^d\times(0,T)),
\]
one has
\[
    -\int_0^T\int_{\mathbb R^d}
    y\,\partial_t\zeta\,d\pi dt
    =
    -\int_0^T\int_{\mathbb R^d}
    \langle\nabla w,\nabla\zeta\rangle\,d\pi dt.
\]
Since
\[
    y\in W^{1,1}_{\mathrm{loc}}
    (0,T;L^1_{\mathrm{loc}}(\mathbb R^d,\pi)),
\]
we may integrate by parts in time. Because \(\zeta\) is compactly supported
in \((0,T)\), no boundary term appears. Hence
\[
    \int_0^T\int_{\mathbb R^d}
    \partial_t y\,\zeta\,d\pi dt
    =
    -\int_0^T\int_{\mathbb R^d}
    \langle\nabla w,\nabla\zeta\rangle\,d\pi dt.
\]

Let \(\theta_\varepsilon\in C^\infty(\mathbb R)\) satisfy
\[
    0\le \theta_\varepsilon\le1,
    \qquad
    \theta_\varepsilon'\ge0,
\]
\[
    \theta_\varepsilon(r)=0
    \quad\text{for }r\le0,
    \qquad
    \theta_\varepsilon(r)=1
    \quad\text{for }r\ge\varepsilon.
\]
Then
\[
    \theta_\varepsilon(w)
    \to
    \mathbf 1_{\{w>0\}}
    =
    \chi
    \qquad\text{a.e. on }\mathbb R^d\times(0,T).
\]

We take
\[
    \zeta_\varepsilon:=\theta_\varepsilon(w)\varphi.
\]
This function is not necessarily smooth, but it is admissible by density.
Indeed,
\[
    \zeta_\varepsilon\in
    L^2_{\mathrm{loc}}(0,T;H^1_c(\mathbb R^d,\pi)),
\]
and
\[
    \nabla\zeta_\varepsilon
    =
    \theta_\varepsilon'(w)\varphi\nabla w
    +
    \theta_\varepsilon(w)\nabla\varphi.
\]
Approximating \(\zeta_\varepsilon\) by smooth compactly supported functions
in \(L^2(0,T;H^1)\), the right-hand side passes to the limit by
\(L^2\)-convergence of the gradients, while the left-hand side passes to
the limit because \(\partial_t y\in L^1_{\mathrm{loc}}\) and
\(\zeta_\varepsilon\) is bounded and compactly supported.

Therefore we may use \(\zeta_\varepsilon\) as a test function. We get
\[
    \int_0^T\int_{\mathbb R^d}
    \partial_t y\,\theta_\varepsilon(w)\varphi\,d\pi dt
    =
    -\int_0^T\int_{\mathbb R^d}
    \left\langle
        \nabla w,
        \nabla\bigl(\theta_\varepsilon(w)\varphi\bigr)
    \right\rangle d\pi dt.
\]
Expanding the gradient gives
\[
    \nabla\bigl(\theta_\varepsilon(w)\varphi\bigr)
    =
    \theta_\varepsilon'(w)\varphi\nabla w
    +
    \theta_\varepsilon(w)\nabla\varphi.
\]
Hence
\[
\begin{aligned}
    \int_0^T\int_{\mathbb R^d}
    \partial_t y\,\theta_\varepsilon(w)\varphi\,d\pi dt
    &=
    -\int_0^T\int_{\mathbb R^d}
    \theta_\varepsilon'(w)\varphi|\nabla w|^2\,d\pi dt  \\
    &\quad
    -\int_0^T\int_{\mathbb R^d}
    \theta_\varepsilon(w)
    \langle\nabla w,\nabla\varphi\rangle\,d\pi dt.
\end{aligned}
\]
Since
\[
    \theta_\varepsilon'(w)\ge0
    \qquad\text{and}\qquad
    \varphi\ge0,
\]
the first term on the right-hand side is nonpositive. Therefore
\[
    \int_0^T\int_{\mathbb R^d}
    \partial_t y\,\theta_\varepsilon(w)\varphi\,d\pi dt
    \le
    -\int_0^T\int_{\mathbb R^d}
    \theta_\varepsilon(w)
    \langle\nabla w,\nabla\varphi\rangle\,d\pi dt.
\]

Letting \(\varepsilon\downarrow0\), we obtain by dominated convergence
\[
    \int_0^T\int_{\mathbb R^d}
    \chi\,\partial_t y\,\varphi\,d\pi dt
    \le
    -\int_0^T\int_{\mathbb R^d}
    \chi\langle\nabla w,\nabla\varphi\rangle\,d\pi dt.
\]
Indeed, on the support of \(\varphi\),
\[
    |\theta_\varepsilon(w)\partial_t y\,\varphi|
    \le
    |\partial_t y|\,|\varphi|,
\]
which is integrable, and
\[
    \left|
    \theta_\varepsilon(w)
    \langle\nabla w,\nabla\varphi\rangle
    \right|
    \le
    |\nabla w|\,|\nabla\varphi|,
\]
which is integrable because \(\nabla w\in L^2_{\mathrm{loc}}\) and
\(\nabla\varphi\in L^\infty_c\).

Using
\[
    \partial_t y^+=\chi\,\partial_t y,
    \qquad
    \nabla w^+=\chi\nabla w,
\]
we get
\[
    \int_0^T\int_{\mathbb R^d}
    \partial_t y^+\,\varphi\,d\pi dt
    \le
    -\int_0^T\int_{\mathbb R^d}
    \langle\nabla w^+,\nabla\varphi\rangle\,d\pi dt.
\]
Finally, since \(\varphi\in C_c^\infty(\mathbb R^d\times(0,T))\),
integration by parts in time gives
\[
    \int_0^T\int_{\mathbb R^d}
    \partial_t y^+\,\varphi\,d\pi dt
    =
    -\int_0^T\int_{\mathbb R^d}
    y^+\,\partial_t\varphi\,d\pi dt.
\]
Therefore
\[
    -\int_0^T\int_{\mathbb R^d}
    y^+\,\partial_t\varphi\,d\pi dt
    \le
    -\int_0^T\int_{\mathbb R^d}
    \langle\nabla w^+,\nabla\varphi\rangle\,d\pi dt.
\]
This proves
\[
    \partial_t y^+\le Lw^+
\]
in the sense of distributions on \(\mathbb R^d\times(0,T)\).
\end{proof}

\subsection{H\"older Continuity of Quasilinear Equations of the Porous Medium Type}
This section is taken from \cite[Chapter 3, Chapter 5, Theorem 15.1, Theorem 16.1]{dibenedetto2012harnack}.

Let $E$ be an open set in $\mathbb{R}^d$ and for $T>0$ let $E_T$ denote the cylindrical domain $E \times(0, T)$. Consider quasilinear, degenerate or singular parabolic partial differential equations of the form

\begin{equation}\label{eq:g51}
    \partial_t u-\operatorname{div} \mathbf{A}(x, t, u, \nabla u)=B(x, t, u, \nabla u) \quad \text { weakly in } E_T
\end{equation}where the functions $\mathbf{A}: E_T \times \mathbb{R}^{d+1} \rightarrow \mathbb{R}^d$ and $B: E_T \times \mathbb{R}^{d+1} \rightarrow \mathbb{R}$ are only assumed to be measurable and subject to the structrue conditions

$$
\left\{\begin{array}{l}
\inner{\mathrm{A}(x, t, u, \nabla u)}{\nabla u} \geq C_o m|u|^{m-1}\norm{\nabla u}^2-C^2|u|^{m+1} \\
\norm{\mathrm{~A}(x, t, u, D u)} \leq C_1 m|u|^{m-1}\norm{\nabla u}+C|u|^{m} \\
|B(x, t, u, D u)| \leq C m|u|^{m-1}\norm{\nabla u}+C^2|u|^{m}
\end{array} \text { a.e. in } E_T\right.
$$where $m>0, C_o$, and $C_1$ are given positive constants, and $C$ is a given nonnegative constant. When $C=0$ the equation is homogeneous.

A function
$$
\begin{aligned}
& u \in \mathcal{C}_{\mathrm{loc}}\left(0, T ; L_{\mathrm{loc}}^2(E)\right),|u|^{\frac{m+1}{2}} \in L_{\mathrm{loc}}^2\left(0, T ; W_{\mathrm{loc}}^{1,2}(E)\right)\quad \text { if } m>1 \\
& u \in \mathcal{C}_{\mathrm{loc}}\left(0, T ; L_{\mathrm{loc}}^{m+1}(E)\right), |u|^m \in L_{\mathrm{loc}}^2\left(0, T ; W_{\mathrm{loc}}^{1.2}(E)\right) \quad\text { if } 0<m<1
\end{aligned}
$$is a local, weak sub(super)-solution to \eqref{eq:g51} if for every compact set $K \subset E$ and every sub-interval $ [t_1,t_2]\subset (0,T]$
\begin{equation*}
    \begin{aligned}
        \left.\int_K u \varphi d x\right|_{t_1} ^{t_2}+\int_{t_1}^{t_2} \int_K\left[-u \partial_t\varphi+\inner{\mathbf{A}\left(x, t, u, \nabla u\right)}{\nabla\varphi}\right] d x d t \\
\leq(\geq) \int_{t_1}^{t_2} \int_K B\left(x, t,,u, \nabla u\right) \varphi d x d t
    \end{aligned}
\end{equation*}
for all nonnegative testing functions
$$
\varphi\in W_{\text {loc }}^{1,2}\left(0, T ; L^2(K)\right) \cap L_{\mathrm{loc}}^2\left(0, T ; W_{0}^{1,2}(K)\right).
$$

In the following, we always assume $m>1$.

Let
$$
\Gamma''=\partial E_T-\bar{E} \times\{T\}
$$denote the parabolic boundary of $E_T$, and for a compact set $K \subset E_T$ introduce the intrinsic, parabolic $m$-distance from $K$ to $\Gamma''$ by
$$
m-\operatorname{dist}(K ; \Gamma'') \stackrel{\text { def }}{=} \inf _{\substack{(x, t) \in K \\(y, s) \in \Gamma''}}\left(\norm{x-y}+\|u\|_{\infty, E_T}^{\frac{m-1}{2}}|t-s|^{\frac{1}{2}}\right) .
$$

\begin{theorem}\label{thm:apx7}
    Let $u$ be a bounded, local, weak solution to the degenerate porous medium type equations \eqref{eq:g51} which satisfy the above structure conditions. Then $u$ is locally H\"older continuous in $E_T$, and there exist constants $\gamma>1$ and $p \in(0,1)$ that can be determined a priori only in terms of the data $\left\{m, d, C_o, C_1, C\right\}$, such that for every compact set $K \subset E_T$,$$
\left|u\left(x_1, t_1\right)-u\left(x_2, t_2\right)\right| \leq \gamma\|u\|_{\infty, E_T}\left(\frac{\norm{x_1-x_2}+\|u\|_{\infty, E_T}^{\frac{m-1}{2}}\left|t_1-t_2\right|^{\frac{1}{2}}}{m-\operatorname{dist}(K ; \Gamma'')}\right)^p
$$for every pair of points $\left(x_1, t_1\right)$, and $\left(x_2, t_2\right) \in K$.
\end{theorem}

\subsection{Schauder Estimate}
The following Schauder estimate is from \cite[Chapter 2, Theorem 6]{schauder-harvard}.

Let us consider a linear operator of the form
$L'=a^{i j}(x, t) D_i D_j+b^i(x, t) D_i+c(x, t)$, and define $Q_R=B_R(0) \times\left(-R^2, 0\right]$, then we have the following Schauder interior estimate.
\begin{theorem}\label{thm:apx8}
   Let $\delta \in(0,1), R>0$ and $k$ be a non-negative integer. Suppose $u(x, t) \in \mathcal{C}^{2+\delta, 1+\delta / 2}\left(Q_{2 R}\right)$ satisfies

$$
\left(\partial_t-L'\right) u=f
$$where $L$ is defined as above and $$
\begin{aligned}
& \left\|D^p a^{i j}\right\|_{\delta, \delta / 2},\left\|D^p b^i\right\|_{\delta, \delta / 2},\left\|D^p c\right\|_{\delta, \delta / 2} \leq \Lambda, \\
& \Lambda^{-1}|\xi|^2 \leq a^{i j} \xi_i \xi_j \leq \Lambda|\xi|^2, \quad \xi \in \mathbf{R}^d, \\
& D^p f \in \mathcal{C}^{\delta, \delta / 2}\left(Q_{2 R}\right),
\end{aligned}
$$
for all multi-index $|p| \leq k$.
Then $D^p u \in C^{2+\delta, 1+\delta / 2}\left(Q_R\right)$ and

$$
\begin{aligned}
& \left\|D^k u\right\|_{2+\delta, 1+\delta / 2 ; Q_R} \\
\leq & C\left(\sum_{|p| \leq k}\left\|D^p f\right\|_{\delta, \delta / 2 ; Q_{2 R}}+\|u\|_{L^{\infty}\left(Q_{2 R}\right)}\right)
\end{aligned}
$$where $C$ depends on $d, \delta, \Lambda, R$.
\end{theorem}

\subsection{Local Boundness of Elliptic Equation}
The following theorem is adapted from \cite[Chapter 4, Theorem 4.1]{han2011elliptic}.

\begin{theorem}\label{thm:apx9}
    Suppose that $u \in H^1\left(B_1(x_0)\right)$ is a subsolution in the following sense:
\begin{equation}\label{eq:re21}
    \int_{B_1(x_0)} \inner{\nabla u}{\nabla\varphi}d\pi(x) \leq \int_{B_1(x_0)} g \varphi d\pi(x)
\end{equation} for any $\varphi \in H_0^1\left(B_1(x_0)\right)$ and $\varphi \geq 0$ in $B_1(x_0)$.
If $g \in L^q\left(B_1(x_0)\right)$, then $u^{+}:=\max\{u,0\} \in L_{\mathrm{loc}}^{\infty}\left(B_1(x_0)\right)$. Moreover, there holds for any $p>0,q\geq d/2$
$$
\sup _{B_{1/2}(x_0)} u^{+} \leq C\left\{\left\|u^{+}\right\|_{L^p\left(B_1(x_0),\pi\right)}+\|g\|_{L^q\left(B_1(x_0),\pi\right)}\right\}
$$where $C=C(d, p, q,\max_{x\in B_1(x_0)}V(x),\min_{x\in B_1(x_0)}V(x))$ is a positive constant.
\end{theorem}
\begin{proof}
This is just an application of \cite[Chapter 4, Theorem 4.1]{han2011elliptic}. Let $a_{ij}(x)=e^{-[V(x)-\max_{B_1(x_0)}V]}\mathrm{I}_d, c=0$ and $f=ge^{-[V(x)-\max_{B_1(x_0)}V]}$, then the constant in \cite[Chapter 4, Theorem 4.1]{han2011elliptic} will be $\lambda=1,\Lambda=e^{[\max_{B_1(x_0)}V-\min_{B_1(x_0)}V]}$. Use \cite[Chapter 4, Theorem 4.1]{han2011elliptic}, we directly have
\begin{equation}
    \begin{aligned}
        \sup _{B_{1/2}(x_0)} u^{+} &\leq C(d,p,q)\Lambda^{c\frac{d}{p}}\left\{\left\|u^{+}\right\|_{L^p\left(B_1(x_0)\right)}+\|g\|_{L^q\left(B_1(x_0)\right)}\right\}\\
    & \leq C(d,p,q)\Lambda^{c\frac{d}{p}}e^{\frac{\max_{B_1(x_0)}V}{p}}\left\{\left\|u^{+}\right\|_{L^p\left(B_1(x_0),\pi\right)}+\|g\|_{L^q\left(B_1(x_0),\pi\right)}\right\}
    \end{aligned}
\end{equation}
here $c$ is a universal constant independent of $d,p,q$.
\end{proof}

\end{document}